\documentclass[10pt]{amsart}

\usepackage[usenames,dvipsnames]{xcolor}

\usepackage[bookmarks,colorlinks=true,citecolor=OliveGreen,linkcolor=RoyalBlue]{hyperref}
\usepackage{amssymb,latexsym,amsthm}
\usepackage{mathtools}
\usepackage{mathabx}

\usepackage{xparse}

\usepackage{bm}

\usepackage{tikz}
\usetikzlibrary{decorations.pathreplacing}
\usetikzlibrary{arrows,shapes,positioning}
\usetikzlibrary{decorations.markings}

\usepackage{thmtools}
\usepackage[capitalise]{cleveref}		

\makeatletter
\def\thmt@refnamewithcomma #1#2#3,#4,#5\@nil{%
  \@xa\def\csname\thmt@envname #1utorefname\endcsname{#3}%
  \ifcsname #2refname\endcsname
    \csname #2refname\expandafter\endcsname\expandafter{\thmt@envname}{#3}{#4}%
  \fi
}
\makeatother

\declaretheorem[numberwithin=section]{theorem}
\declaretheorem[sibling=theorem]{proposition}

\declaretheorem[sibling=theorem]{corollary}

\declaretheorem[sibling=theorem]{lemma}

\declaretheorem[sibling=theorem,style=definition]{definition}

\declaretheorem[sibling=theorem,style=remark]{remark}
\declaretheorem[sibling=theorem,style=remark]{example}

\newcommand{\Cyl}[1]{\,\,\,\ldbrack{#1}\rdbrack}
\newcommand{\seq}[1]{{\left\langle{#1}\right\rangle}}

\newcommand{\rest}[1]{\!\! \upharpoonright_{#1}}
\newcommand{\tth}{{}^{\textup{th}}}
\newcommand{\conc}{\hat{\,\,}}
\newcommand{\andd}{\,\,\,\&\,\,\,}

\DeclareMathOperator{\dom}{dom}

\DeclareMathOperator{\id}{id}

\newcommand{\converge}{\!\!\downarrow}

\newcommand{\emptystring}{{\seq{}}}

\newcommand{\w}{\omega}
\newcommand{\s}{\sigma}
\newcommand{\vphi}{\varphi}

\renewcommand{\le}{\leqslant}
\renewcommand{\ge}{\geqslant}

\renewcommand{\preceq}{\preccurlyeq}
\renewcommand{\succeq}{\succcurlyeq}
\newcommand{\nle}{\nleqslant}

\newcommand{\Tur}{\textup{\scriptsize T}}

\newcommand{\PP}{{\mathbb{P}}}
\newcommand{\QQ}{{\mathbb{Q}}}

\newcommand{\GG}{{{G}}}

\newcommand{\pp}{{\mathbf{p}}}
\newcommand{\qq}{{\mathbf{q}}}
\newcommand{\rr}{{\mathbf{r}}}
\renewcommand{\ss}{{\mathbf{s}}}

\newcommand{\Quick}{{\mathcal{Q}}}

\newcommand{\DNC}{{\textup{DNC}}}

\newcommand\force{\Vdash}

\newcommand{\CC}{\mathcal{C}}

\newcommand{\+}[1]{{\boldsymbol{#1}}}

\newcommand{\length}[1]{\ell(#1)}

\newcommand{\chop}[1]{#1\!\!\downarrow}
\newcommand{\cchop}[1]{#1\downarrow}

\newcommand{\interpret}[1]{\left\lfloor #1 \right\rfloor}

\newcommand{\project}[3]{\pi^{#1}_{#2}(#3)}

\newcommand{\Split}[4]{#1\text{\textup{-Sp}}^{#2}_{#3}(#4)}

\DeclareDocumentCommand\Strings{ g }
	{ \IfNoValueF{#1}{\left(} \w^{<\w}
		\IfNoValueF{#1}{\right)^{#1}}
	}

\DeclareDocumentCommand\Baire{ g }
	{ \IfNoValueF{#1}{\left(} \w^{\w}
		\IfNoValueF{#1}{\right)^{#1}}
	}

\newcommand {\revmathfont}[1]{\mathsf{#1}}
\newcommand {\naughtsys}[1]{{\revmathfont{#1}}_0}

\newcommand\WKLz{\naughtsys{WKL}}

\newcommand\WWKLz{\naughtsys{WWKL}}
\newcommand\DNRz{\naughtsys{DNR}}

\title{DNR and incomparable Turing degrees}

\author{Minzhong Cai}
\address{Department of Mathematics, Dartmouth College, Hanover, NH 03755, USA}
\email{Mingzhong.Cai@dartmouth.edu}

\author{Noam Greenberg}
\address{School of Mathematics Statistics and Operations Research, Victoria University of Wellington, P.O.~Box 600, Wellington, New Zealand}
\email{greenberg@msor.vuw.ac.nz}
\urladdr{\url{http://homepages.mcs.vuw.ac.nz/~greenberg/}}

\author{Michael McInerney}
\address{School of Mathematics Statistics and Operations Research, Victoria University of Wellington, P.O.~Box 600, Wellington, New Zealand}
\email{michael.mcinerney@msor.vuw.ac.nz}

\thanks{Cai was supported by NSF Grant DMS-1458061; Greenberg was supported by the Marsden Fund and a Rutherford Discovery Fellowship from the Royal Society of New Zealand, and by a Turing Research Fellowship ``Mind, Mechanism and Mathematics'' from the John Templeton Foundation.}

\begin{document}

\begin{abstract}
We construct an increasing $\w$-sequence $\seq{\+{a}_n}$ of Turing degrees which forms an initial segment of the Turing degrees, and such that each~${\+{a}_{n+1}}$ is diagonally noncomputable relative to $\+a_n$. It follows that the~$\mathsf{DNR}$ principle of reverse mathematics does not imply the existence of Turing incomparabile degrees.
\end{abstract}

\maketitle

\section{Introduction}

In~\cite{KuceraSlaman:TuringIncomparabilitScottSets}, Ku\v{c}era and Slaman solved a long-standing open problem by showing that no Scott set is ``hourglass-shaped'': if $\mathcal S$ is a Scott set of reals and $x\in \mathcal S$ is noncomputable then there is some $y\in \mathcal S$ which is Turing incomparable with~$x$. In other words, Turing incomparability holds in every $\w$-model of the system $\WKLz$ (weak K\"{o}nig's lemma) --- the system ensuring the existence of completions of Peano Arithmetic. This was improved by Conidis~\cite{Conidis:WWKL} to show that the statement holds in~$\w$-models of the weaker system $\WWKLz$, the system which ensures the existence of a Martin-L\"{o}f random set. A prominent system below $\WWKLz$ is $\DNRz$, the system which ensures the existence of a \emph{diagonally noncomputable function} (DNC): a function $f\colon \w\to\w$ which disagrees with the Turing Jump function (for example $J(e)= \vphi_e(e)$) on the latter's domain. These functions were introduced by Jockusch~\cite{Jockusch:DNR}, who showed that their Turing degrees conicide with the degrees of fixed-point-free functions, those functions which escape the recursion (fixed-point) theorem. The two systems~$\WWKLz$ and $\DNRz$ were first separated by Ambos-Spies et al.~\cite{A-S_K-H_L_S:DNRandWWKL}. They used a tame version of the ``bushy tree'' forcing technique first used by Kumabe in his construction of a fixed-point-free minimal degree (see~\cite{KumabeLewis}). In this paper we extend this technique to show:

\begin{theorem} \label{thm:main}
	There is an initial segment $\+a_1 < \+a_2 < \+a_3 < \cdots$ of the Turing degrees such that each~$\+a_{n+1}$ is a~$\DNC$ degree relative to $\+a_n$.
\end{theorem}

\begin{corollary} \label{cor:main}
	The system~$\DNRz$ does not imply Turing incomparability, in fact it does not imply the existence of a pair of Turing incomparable reals.
\end{corollary}

We prove \cref{thm:main} in four steps. The third step (in \cref{sec:step_n}) provides the construction, for each~$n<\w$, of an initial segment $\+a_1 < \cdots < \+a_n$ of the desired infinite sequence $\seq{\+a_k}$. The fourth and last step (in \cref{sec:proof_of_the_main_theorem}) shows how to string these constructions together and so prove \cref{thm:main}. The first two steps serve as an introduction to the construction of \cref{sec:step_n}. In \cref{sec:step_one} we recast Kumabe's construction in the language of forcing that we subsequently use. In \cref{sec:step_two} we discuss the case $n=2$ (the construction of a minimal~$\DNC$ degree~$\+a_1$ and a strong minimal cover~$\+a_2$ of~$\+a_1$ which is $\DNC$ relative to~$\+a_1$).

\subsection{Quick-growing functions} \label{subsec:quick}

Below we use trees (or tree systems) which are fairly ``bushy'' but associated with them we will have sets of ``bad'' strings which we want to avoid. In the first step we use infinite trees and for example declare every string which is not DNC to be bad. We then extend the bad set of strings when we force divergence or force a functional to be constant on a tree. We cannot simply remove the bad strings from the tree because the trees will be computable whereas the set of bad strings will be c.e. To ensure that most strings are not bad, and that the construction can proceed, we will require that the tree is~$h$-bushy and that the bad set of strings is~$b$-small above the stem of the tree, where~$h$ grows much more quickly than the order-function~$b$. Here we discuss the notion of relative quickness that we will use.

For an equivalence notion of rate of growth we close under relative elementary recursive functions. (We could use relative primitive recursive functions but this is not needed.) For any order function~$h$ one defines the class of order functions which are obtained from~$h$ using a list of rules such as substitution and bounded summation and multiplication.

We are only concerned with rates of growth. If~$h$ grows sufficiently quickly then~$g$ is bounded by a function elementary in~$h$ if and only if it is dominated by an iterated composition of~$h$ with itself. In particular, the elementary recursive functions are those which are bounded by iterated exponentials.

It will be convenient to consider functions that may be undefined on a finite initial segment of~$\w$.

\begin{definition} \label{def:quick-growing-functions}
	Let~$\Quick$ denote the collection of nondecreasing computable functions~$h\colon \w \to  [2,\w)$ satisfying $h(n)\ge 2^{n}$ for all~$n$. \end{definition}

For~$h\in \Quick$ let $h^{(1)}=h$ and for~$k\ge 1$, $h^{(k+1)} = h\circ h^{(k)}$. For two functions~$h$ and~$g$ in~$\Quick$ we say that~$h$ \emph{majorises}~$g$ if $h(n)\ge g(n)$ for all~$n$ (and write $h\ge g$). We say that $h\ge g$ \emph{above~$m$} if $h(n)\ge g(n)$ for all $n\ge m$. We say that~$h$ \emph{dominates}~$g$ if $h\ge g$ above some~$m$ (and write $h\ge^* g$).

\medskip

We will use the fact that iterated exponentials of~$h$ are dominated by iterates of~$h$. For example:

\begin{example} \label{lem:bounded_products}
	Let~$h\in \Quick$. Let $g(n) =  \prod_{m\in [0,n)} h(m)$. Then $g\le^* h^{(3)}$. For $g \le h^h$ whereas $h^{(2)} \ge  2^h$ and $h^{(3)}\ge 2^{2^h}$, and $2^{2^h}\ge^* h^h$.
\end{example}

\begin{definition} \label{def:much_larger}
	Let $h,g\in \Quick$. We say that~$h$ \emph{dominates the iterates of~$g$ uniformly}, and write $h\gg g$, if there is a computable sequence~$\seq{d_k}$ such that 
	for all~$k\ge 1$, $h\ge g^{(k)}$ on the interval $(d_k,\w)$.
\end{definition}


 The relation~$\gg$ on~$\Quick$ is transitive. Indeed if $h\gg g$, $h'\ge^* h$ and $g\ge^* g'$ then $h'\gg g'$. Further, $h\gg g^{(k)}$ for all~$k$, and so for example $h\gg 2^g$.

\smallskip

The following density lemma will be used to keep extending conditions.

\begin{lemma} \label{lem:density_of_gg}
	For all $h,g\in \Quick$ such that $h\gg g$ there is some $f\in \Quick$ such that $h \gg f \gg g$.
\end{lemma}

\begin{proof}
The idea is to gradually let~$f$ copy~$g^{(k)}$. If~$f$ is bounded by~$g^{(k)}$ for a long time, then for a shorter time we can ensure that $f^{(k)}$ is bounded by $g^{(k^2)}$, so we do this until the point where~$h$ starts to majorise~$g^{((k+1)^2)}$, and only then start copying~$g^{(k+1)}$.

Since~$g$ is nondecreasing and dominates the identity, each~$g^{(k)}$ is nondecreasing and $g^{(k+1)}\ge g^{(k)}$.

Let $k\ge 1$, $e\ge 0$ and let~$f$ be a function. Suppose that $f \le g^{(k)}$ on the interval $[0,g^{(k^2)}(e)]$ (actually the interval $[0,g^{(k^2-k)}(e)]$ will suffice). Then $f^{(k)}\le g^{(k^2)}$ on the interval $[0,e]$: by induction on $j\le k$ we see that $f^{(j)}\le g^{(kj)}$ on the interval $[0,g^{(k(k-j))}(e)]$.

Let~$\seq{d_k}$ witness that $h\gg g$. We may assume that~$\seq{d_k}$ is nondecreasing.

We define a computable sequence $-1= a_0 \le a_1 \le \cdots$ and then define~$f$ by letting $f = g^{(k)}$ on the interval $(a_{k-1},a_k]$. So the sequence $\seq{a_{k-1}}$ witnesses that $f \gg g$. But also $f\le g^{(k)}$ on the interval $[0,a_k]$ for all $k\ge 1$. So we let $a_k = g^{(k^2)}(d_{(k+1)^2})$. This ensures that $f^{(k)}\le g^{(k^2)}$ on $[0,d_{(k+1)^2}]$, which in turn shows that $h\ge f^{(k)}$ on the interval $(d_{k^2},d_{(k+1)^2}]$. Since $f\in \Quick$, $f^{(m)}\ge f^{(k)}$ if $m\ge k$, so the sequence $\seq{d_{k^2}}$ witnesses that $h\gg f$.
\end{proof}

\subsection{Other notation and conventions} \label{subsec:conventions}

A \emph{string} is a finite sequence of natural numbers, an element of~$\w^{<\w}$. If~$\s$ is a string then we let $\s^{\preceq}$ be the collection of strings which extend~$\s$, and $[\s]^\prec$ be the set of elements of Baire space~$\w^\w$ which extend~$\s$. If~$C$ is a set of strings then $C^\preceq = \bigcup_{\s\in C} \s^\preceq$ and so $[C]^\prec = \bigcup_{\s\in C}[\s]^\prec$.

We may assume that for any Turing functional~$\Gamma$ and for any string~$\tau$, the domain of~$\Gamma(\tau)$ is downwards closed. Thus~$\Gamma$ determines a monotone computable map $\tau\mapsto \Gamma(\tau)$ from strings to strings, which induces a partial computable function on Baire space: $\Gamma(x) =  \bigcup \{ \Gamma(\tau)\,:\,\tau\prec x\}$.

We let lowercase Greek letters denote strings, lowercase Roman letters denote elements of Baire space, and uppercase Roman letters denote sets of strings and sometimes subsets of Baire space.

\subsection{Compactness, splittings and computability} \label{subsec:compactness_splitting}

\begin{definition} \label{def:computably_bounded}
	A subset~$X$ of Baire space is \emph{computably bounded} if some computable function majorises every element of~$X$.
\end{definition}

Every computably bounded and closed subset of Baire space is compact.

\smallskip

The following is well-known.

\begin{lemma} \label{lem:1-1_or_constant_on_compact_space}
	Let~$X\subseteq \w^\w$ be $\Pi^0_1$ and computably bounded; let $f\colon X\to 2^\w$ be a computable function.
	\begin{itemize}
		\item If~$f$ is constant on~$X$ then this constant value is computable.
		\item If~$f$ is 1-1 on~$X$ then for all $x\in X$, $x\equiv_\Tur f(x)$.
	\end{itemize}
\end{lemma}

\begin{proof}
	Suppose that~$f$ is constant on~$X$; let $f[X] = \{y\}$. The fact that~$X$ is computably bounded implies that the set of~$\alpha\in 2^{<\w}$ such that $X= f^{-1}\left[[\alpha]^\prec\right]$ is c.e.; this is the set of initial segments of~$y$, so~$y$ is computable.
		
	\smallskip
	
	Suppose that~$f$ is 1-1 on~$X$. Let~$Y = f[X]$. Then~$Y$ a $\Pi^0_1$ subset of~$2^\w$ and~$f$ is a homeomorphism between~$X$ and~$Y$. And~$f^{-1}$ is computable: the set of pairs $(\alpha,\tau)$ such that $[\alpha]^\prec\cap Y\subseteq f\left[[\tau]^\prec\right]$ is c.e.
\end{proof}	

If $X\subseteq \Baire{2}$ and $x\in \Baire$ we let $X_x = \{y\,:\, (x,y)\in X\}$.

\begin{lemma} \label{lem:1-1_on_sections_of_compact}
	Let~$X\subseteq \Baire{2}$ be $\Pi^0_1$ and computably bounded. Let $f\colon X\to 2^\w$ be computable and suppose that the collection of sets $f[X_x]$ for $x\in \dom X$ are pairwise disjoint. Then for all $(x,y)\in X$, $x\le_\Tur f(x,y)$.
\end{lemma}

\begin{proof}
	For $\tau\in \Strings$ let $X_\tau = \bigcup_{x\in [\tau]^\prec} X_x$. The set of pairs $(\tau,C)$ where $C\subseteq 2^\w$ is clopen and $f[X_\tau] = C\cap f[X]$ is c.e.
\end{proof}

\subsection{Forcing with closed sets} \label{subsec:forcing_with_closed_sets}

\begin{definition} \label{def:forcing_with_closed_sets}
	Let~$\PP$ be a notion of forcing. Suppose that with each condition $\pp\in \PP$ we associate a closed subset~$X^\pp$ of Baire space. We call this assignment \emph{acceptable} if:
\begin{enumerate}
	\item[(a)] for all $\pp\in \PP$, $X^\pp$ is nonempty;
	\item[(b)] if $\qq$ extends~$\pp$ then $X^\qq\subseteq X^\pp$; and
	\item[(c)] for every~$m$, the set of conditions~$\pp\in \PP$ such that $X^\pp\subseteq [\s]^{\prec}$ for some string~$\s$ of length~$m$ is dense in~$\PP$.	
\end{enumerate}	
\end{definition}

(Below we will consider finite powers $(\w^\w)^n$ of Baire space, but these are of course effectively isomorphic to Baire space.)

\smallskip

Recall the \emph{Borel codes} for Borel subsets of Baire space. These can be identified with propositional sentences in~$L_{\w_1,\w}$. To be precise:
\begin{itemize}
	\item Every finite set of strings~$C$ is a Borel code;
	\item If~$C$ is a Borel code then $\lnot C$ is a Borel code;
	\item If $\CC$ is a countable set of Borel codes, then $\bigvee \CC$ and $\bigwedge \CC$ are Borel codes.
\end{itemize}

The semantics are obvious (a finite set of strings~$C$ defines the set $[C]^\prec$); if~$C$ is a Borel code then we let~$\interpret{C}$ be the Borel subset of Baire space defined by~$C$.

\smallskip

Suppose that $\PP$ is a notion of forcing equipped with an acceptable assignment of closed sets $X^\pp$.
We define the forcing relation $\pp\force C$ between conditions in~$\PP$ and Borel codes~$C$. We start with strong forcing.

\begin{definition} \label{def:strong_forcing}
	Let~$C$ be a Borel code and let $\pp\in \PP$. We say that $\pp$ \emph{strongly forces} $C$ if $X^{\pp}\subset \interpret{C}$. We write $\pp\force^* C$.
\end{definition}

Now by recursion on Borel codes~$C$ we define forcing.
\begin{itemize}
	\item For a finite set of strings~$D$, $\pp\force D$ if the collection of conditions which strongly force~$D$ is dense below~$\pp$.
	\item $\pp\force \lnot C$ if no extension of~$\pp$ forces~$C$.
	\item $\pp\force \bigwedge \CC$ if $\pp\force C$ for all $C\in \CC$.
	\item $\pp\force \bigvee \CC$ if the set of conditions which force some element of~$\CC$ is dense below~$\pp$.
\end{itemize}

The basic properties of forcing hold.

\begin{lemma} \label{lem:forcing_behaves}
	Let $\pp\in \PP$ and let~$C$ be a Borel code.
	\begin{enumerate}
		\item No condition forces both~$C$ and $\lnot C$.
		\item The set of conditions which decide~$C$ is dense in~$\PP$.
		\item If~$\qq$ extends~$\pp$ and $\pp\force C$ then $\qq\force C$.
		\item If the set of conditions which force~$C$ is dense below~$\pp$ then $\pp\force C$.
	\end{enumerate}
\end{lemma}

Forcing equals truth. It will be convenient to consider directed subsets of~$\PP$ rather than filters; of course the upwards closure of a directed set is a filter, so we can always pass to filters without adding information. Genericity for directed sets is defined using desne \emph{open} sets: dense subsets of~$\PP$ which are closed downwards (closed under taking extensions). Note that the dense sets of conditions mentioned above are all open.

Suppose that $\GG\subset \PP$ is a directed set. If~$\GG$ meets all of the dense open sets of conditions guaranteed by~(c) above, then $\bigcap_{\pp\in \PP} X^{\pp}$ is a singleton that we denote by~$\{x^\GG\}$. (This uses the completeness of Baire space; we do not need the sets~$X^\pp$ to be compact.)

In the rest of the paper, the statement ``for all sufficiently generic $\GG\subset \PP$ ...'' means: there is a countable collection~$\mathcal D$ of dense open subsets of~$\PP$ such that for every directed subset of~$\PP$ meeting all the sets in~$\mathcal D$, ...

\begin{lemma} \label{lem:forcing_equals_truth}
	Let~$C$ be a Borel code. If $\GG\subset \PP$ is a sufficiently generic directed set then $x^\GG\in \interpret{C}$ if and only if $\pp\force C$ for some~$\pp\in \GG$.
\end{lemma}

\begin{proof}
	First note that if $\pp\in \GG$ and $\pp\force^* C$ then $x^\GG\in \interpret{C}$. On the other hand, suppose that~$D$ is a finite set of strings, and suppose that~$x^\GG\in [D]^\prec$: there is some~$\tau\prec x^\GG$ such that $\tau \in  D$. By assumption, there is some string~$\eta$ of length~$|\tau|$ and some~$\pp\in \GG$ such that~$X^\pp\subseteq [\eta]^\prec$. Then~$\eta=\tau$, and so $\pp\force^* D$, which implies that $\pp\force D$.
	
	The rest of the argument follows the usual proof of the equivalence of forcing and truth for generic filters.	
\end{proof}

Since every condition can be extended to a sufficiently generic directed set, we conclude:

\begin{corollary} \label{cor:forcing_and_generics}
	Let~$\pp\in \PP$ and let~$C$ be a Borel code.
	\begin{enumerate}
		\item $\pp\force C$ if and only if for every sufficiently generic directed set~$\GG$, if $\pp\in \GG$ then $x^\GG\in \interpret{C}$.
		\item If $\interpret{C} \subseteq  \interpret{C'}$ and $\pp\force C$ then $\pp\force C'$.
		\item If $\pp\force^* C$ then $\pp\force C$.
	\end{enumerate}
\end{corollary}

In light of~(2) we write $\pp\force x^\GG\in  A$ when~$A$ is a Borel subset of Baire space, rather than a code for such a set. 

\subsection{Simplified iterated forcing} \label{subsec:simplified_iterated_forcing}

We give a not-completely-standard definition for restriction maps between notions of forcing.

\begin{definition} \label{def:restriction_maps}
	Let~$\PP$ and~$\QQ$ be partial orderings. A \emph{restriction map} from~$\QQ$ to~$\PP$ is an order-preserving map $i$ from~$\QQ$ to~$\PP$ such that for all $\qq\in \QQ$, the image of $\QQ(\le q)$ (the set of extensions of~$\qq$) under~$i$ is dense below~$i(\qq)$.
\end{definition}
That is, for all $q\in \QQ$ and $p\le i(q)$ there is some $r\le q$ in~$\QQ$ such that $i(r)  \le p$.

\begin{lemma} \label{lem:restriction_maps_and_generics}
	Let $i\colon \QQ\to \PP$ be a restriction map.
\begin{enumerate}
	\item If $G\subset \QQ$ is a directed set then $i[G]\subset \PP$ is a directed set.
	\item If $D\subseteq \PP$ is dense and open then $i^{-1}[D]\subseteq \QQ$ is dense and open.
\end{enumerate}
\end{lemma}

Hence for any collection~$\mathcal D$ of dense open subsets of~$\PP$ there is a collection~$\mathcal E$ of dense open subsets of~$\QQ$ such that if~$G\subset\QQ$  is a directed set which meets every set in~$\mathcal E$, then $i[G]$ is a directed set which meets every set in~$\mathcal D$. In other words, if $G$ is sufficiently generic then so is~$i[G]$.

\smallskip

Suppose that $\PP$ and~$\QQ$ have acceptable assignments of closed sets $X^\pp\subseteq \Baire$ for $\pp\in \PP$ and $Y^\qq\subseteq \Baire{2}$ for $\qq\in \QQ$. Suppose that $i\colon \QQ\to \PP$ is a restriction map and further that for all $\qq\in \QQ$, $X^{i(\qq)} \supseteq \dom Y^{\qq}$. Let~$G\subset \QQ$ be sufficiently generic; we denote the generic pair of reals by $(x^G,y^G)$. Then $x^{i[G]} = x^G$.

\subsection{The plan} 
\label{sub:the_plan}

To prove \cref{thm:main}, for each~$n<\w$ we define a notion of forcing~$\PP_n$ which adds an initial segment of the degrees of length~$n$, each degree DNC relative to the one below it. We then show that there are restriction maps from each~$\PP_n$ to~$\PP_{n-1}$. This will allow us to obtain generic $\GG_n\subset \PP_n$ which are coherent, from which we will obtain the desired $\w$-sequence of degrees.


\section{A DNC minimal degree} \label{sec:step_one}

Khan (see~\cite{MillerKhan}) showed that for any~$x\in 2^\w$ there is a $\DNC^x$ function of minimal Turing degree. He presented an elaboration on the Kumabe-Lewis construction using the language of forcing in computability (rather than give an explicit construction). The extra complication is due to the fact that the set of strings which are not $\DNC^x$ is c.e.\ in~$x$, rather than merely c.e. We have no access to this set when defining the computable trees. For this reason Khan needs to use trees with terminal elements (and the set of terminal elements is co-c.e.\ but not computable).

In this section we present a proof of the original Kumabe-Lewis theorem using the language of forcing. We use c.e.\ sets of bad strings and trees with no terminal elements.

\subsection{Trees and forests} \label{subsec:trees}

We follow~\cite{A-S_K-H_L_S:DNRandWWKL,GreenbergMiller:DNC_and_Hausdorff,MillerKhan} and use trees which are sets of strings rather than function trees (as in~\cite{Cai:2-minimal_GH1,KumabeLewis}). We localise to basic clopen sets.

Recall that for a string~$\s$, $\s^\preceq$ is the set of strings extending~$\s$. A \emph{tree above~$\s$} is a nonempty subset of~$\s^{\preceq}$ which is closed in~$\s^\preceq$ under taking initial segments. If~$A$ is a finite prefix-free set of strings then a \emph{forest above~$A$} is a set~$T\subseteq A^\preceq$ such that for all $\s\in A$, $T\cap \s^\preceq$ is a tree above~$\s$. In particular we require that $A\subseteq T$. When we just say ``tree'' we mean a tree above~$\s$ for some~$\s$; the string~$\s$ will usually be clear from the context or unimportant. The same holds for forests. We will mostly only use finite forests, but will use both finite and infinite trees.

Let~$T$ be a forest and let $\tau\in T$. An \emph{immediate successor} of~$\tau$ on~$T$ is a string $\tau'\succ \tau$ on~$T$ such that $|\tau'| = |\tau|+1$. A \emph{leaf} of a forest~$T$, also known as a \emph{terminal} element of~$T$, is a string on~$T$ which has no proper successors on~$T$.

A \emph{subtree} of a tree~$T$ is a subset $S\subseteq T$ which is a tree. Note that the stem of~$S$ may equal the stem of~$T$, or properly extend the stem of~$T$. If~$T$ is a tree and~$\tau\in T$ then the full subtree of~$T$ above~$\tau$ is $T\cap \tau^\preceq$, the set of strings on~$T$ which extend~$\tau$.

If~$T$ is a tree above~$\s$ then~$[T]$ is the set of infinite paths of~$T$, the set of~$x\in \w^\w$ such that $x\rest{n}\in T$ for all~$n\ge |\s|$. This is a closed subset of~$\w^\w$. Recall that~$[\s]^\prec$ is the set of extensions of~$\s$ in Baire space; in our notation, $[\s]^\prec = \left[\s^\preceq\right]$.

A tree~$T$ is \emph{bounded} by a function~$h$ if for all~$\tau \in T$, $\tau(n) < h(n)$ for all $n\le |\tau|$. It is \emph{computably bounded} if~$h$ can be taken to be computable. If~$T$ is computably bounded then so is~$[T]$ (\cref{def:computably_bounded}).

\subsection{Bushy notions of largeness} \label{subsec:bushy}

The basic notions of ``bushiness'' were extended from constant bounds to order functions, see \cite{Cai:PhD,MillerKhan}. We recall the definitions and basic properties. A \emph{bounding function} is a computable function from~$\w$ to~$[2,\w)$.

\begin{definition} \label{def:bushy}
Let~$T$ be a forest above a finite prefix-free set of strings~$A$; let $h$ be a bounding function. We say that~$T$ is \emph{$h$-bushy} if every nonterminal $\tau\in T$ has at least~$h(|\tau|)$ many immediate successors on~$T$.
\end{definition}

Note that for the notion of bushiness, only the values of~$h$ for $n\ge |\s|$ matter.

\begin{definition} \label{def:large}
	Let~$A$ be a finite prefix-free set of strings and let~$B$ be a set of strings. Let~$h$ be a bounding function. The set~$B$ is \emph{$h$-big above~$A$} if there is a finite forest~$T$ above~$A$ which is $h$-bushy, all of whose leaves are elements of~$B$.

	If~$A$ is an infinite set of strings then we say that~$B$ is $h$-big above~$A$ if $B$ is $h$-big above every finite, prefix-free subset of~$A$.
	
	If~$B$ is not $h$-big above~$A$ then we say it is \emph{$h$-small} above~$A$.
\end{definition}

If~$A$ is a singleton~$\{\s\}$ then we say that~$B$ is $h$-big (or $h$-small) above~$\s$. If $A\subseteq B$ then~$B$ is $h$-big above~$A$ for all bounding functions~$h$. A set~$B$ is $h$-big above~$A$ if and only if the set of minimal strings in~$B$ is $h$-big above~$A$. We thus often use the notion for either prefix-free sets of strings, or for \emph{open} sets of strings -- those that are upwards closed (closed under taking extensions). Also note that sometimes we do not assume that~$B$ only contains extensions of~$A$, but of course for this notion it suffices to look at $B\cap A^{\preceq}$.

\smallskip

The following remark is trivial. Its generalisations in later sections will be less so.

\begin{remark} \label{rmk:canonical_bushy_witness_is_a_subtree}
	Suppose that~$B$ is a set of strings, $h$-big above~$A$, and that~$A,B\subseteq T$ for some tree~$T$. Then any forest~$S$ which witnesses that~$B$ is $h$-big above~$A$ is a subset of~$T$.
\end{remark}

The basic combinatorial properties of this notion of largeness have been repeatedly observed \cite{Jockusch:DNR,KumabeLewis,GreenbergMiller:DNC_and_Hausdorff,MillerKhan}.

\begin{lemma}[Big subset property]	\label{lem:bigsubset}
Let~$h$ and~$g$ be bounding functions. Let~$B$ and~$C$ be sets of strings, let~$\s$ be a string, and suppose that $B\cup C$ is $(h+g)$-big above~$\s$. Then either~$B$ is~$h$-big above~$\s$ or~$C$ is $g$-big above~$\s$.
\end{lemma}

Here it is important that we work above a single string~$\s$ and not above any finite~$A$.

\begin{proof}
Let~$T$ be a tree which witnesses that~$B\cup C$ is $(h+g)$-big above~$\s$. Label a leaf~$\tau$ of~$T$ ``B'' if it is in~$B$, and~``C'' otherwise. Now if $\rho\in T$ and all immediate successors of~$\rho$ have been labelled then since~$\rho$ has at least $h(|\rho|)+ g(|\rho|)$ immediate successors on~$T$, either at least~$h(|\tau|)$ of these are labelled ``B'' or at least $g(|\tau|)$ of them are labelled~``C''. In the first case label $\rho$ ``B'', in the other label it~``C''. Eventually~$\sigma$ is labelled. If~$\sigma$ is labelled ``B'' then set of $\rho\in T$ labelled ``B'' form a tree which witnesses that~$B$ is $h$-big above~$\s$; and similarly if~$\s$ is labelled~``C''.
\end{proof}

\begin{lemma}[Concatenation property] \label{lem:concatentation_property}
Let~$h$ be a bounding function. Let~$A,B$ and~$C$ be sets of strings. Suppose that~$B$ is $h$-big above~$A$, and that $C$ is $h$-big above every $\tau\in B$. Then~$C$ is $h$-big above~$A$.
\end{lemma}

\begin{proof}
	Let~$A'$ be a finite, prefix-free subset of~$A$. Let~$T$ be a forest which witnesses that~$B$ is $h$-big above~$A'$. For a leaf~$\tau$ of~$T$ let $R_\tau$ be a tree which witnesses that~$C$ is $h$-big above~$\tau$. Then $T\cup \bigcup R_\tau$, where $\tau$ ranges over the leaves of~$T$, witnesses that~$C$ is $h$-big above~$A'$.
\end{proof}

The concatenation property will sometimes be used to recursively build bushy trees meeting infinitely many big sets. Again the following are fairly immediate; their generalisations in the next sections will be less so.

\begin{definition} \label{def:end-extensions}
A forest~$R$ is an \emph{end-extension} of a forest~$S$ if every string in $R\setminus S$ extends some leaf of~$S$. 	
\end{definition}

(This is not the same as the usual definition for partial orderings, but under the usual definition, any tree extension is an end-extension.) The argument proving the concatenation is broken up to show:

\begin{lemma} \label{lem:one_step:extended_concatenation}
	Let~$A,B,C$ be sets of strings, and let~$h$ be a bounding function.
	\begin{enumerate}
		\item Suppose that~$C$ is $h$-big above every~$\tau\in B$. Then~$C$ is $h$-big above~$B$.
		\item Suppose that~$A$ is prefix-free and finite; suppose that~$B$ is $h$-big above~$A$ and that $C$ is~$h$-big above~$B$. Then any forest which witnesses that~$B$ is $h$-big above~$A$ has an end-extension which witnesses that~$C$ is $h$-big above~$A$.
	\end{enumerate}
\end{lemma}

\begin{remark} \label{rmk:one_step:we_assume_that_big_sets_are_leaves_of_trees}
	Throughout, we will assume that whenever we are given a set of strings which is guaranteed to have some largeness property, then this set is the set of leaves of a forest witnessing this property. For example, suppose that we are given a set~$B$ which is~$h$-big above some~$\s$. We will assume, often without mentioning it, that~$B$ is finite, that it is prefix-free, and that every string in~$B$ extends~$\s$.
\end{remark}

\subsection{The notion of forcing and the generic} \label{subsec:step_one:the_notion}

Let~$B_{\DNC}$ be the set of strings~$\s$ that are not initial segments of diagonally noncomputable functions: $\s(e) = J(e)\converge$ for some $e<|\s|$, where~$J$ is a fixed universal jump function, for example $J(e) = \vphi_e(e)$.

Let~$T$ be a tree. We say that a set of strings~$B\subseteq T$ is \emph{open in $T$} if it is upwards closed in~$T$: if $\s\in B$ and $\tau\succeq\s$ is in~$T$ then $\tau\in B$.

\smallskip

We let~$\PP_1$ be the set of tuples $\pp = (\s^\pp,T^\pp,B^\pp,h^\pp,b^\pp)$ satisfying:
\begin{enumerate}
	\item $T^\pp$ is a computably bounded, computable tree above~$\s^\pp$ with no leaves.
	\item $h^\pp\in \Quick$ and $T^\pp$ is~$h^\pp$-bushy.
	\item $B^\pp\subset T^\pp$ is c.e.\ and open in~$T^\pp$, and $B^\pp\supseteq B_{\DNC}\cap T^\pp$.
	\item $b^\pp\in \Quick$ and $B^\pp$ is $b^\pp$-small above~$\s^\pp$.
	\item $h^\pp \gg b^\pp$ and $h^\pp\ge b^\pp$ above $|\s^\pp|$.
\end{enumerate}

\begin{lemma} \label{lem:OneStep:nonempty}
	$\PP_1$ is nonempty.
\end{lemma}

\begin{proof}
	The set~$B_{\DNC}$ is c.e.\ and is $2$-small above the empty string~$\emptystring$. Fix some $b\in \Quick$ (and recall that $b\ge 2$); and find some $h\in \Quick$ such that $h\gg b$ and $h\ge b$ (for example $h(n) = b^{(n+1)}(n)$). Recall that $h^{<\w}$ is the set of~$h$-bounded strings. Then $\pp = (\emptystring, h^{<\w},B_{\DNC}\cap h^{<\w}, h,b)$ is a condition in~$\PP_1$.
\end{proof}

We define a partial ordering on~$\PP_1$ as follows. A condition~$\qq$ extends a condition~$\pp$ if $\s^\pp\preceq \s^\qq$, $T^\qq$ is a subtree of~$T^\pp$, $B^\pp \cap T^\qq\subseteq B^\qq$, and $h^\qq\le h^\pp$  and $b^\qq\ge b^\pp$ above~$|\s^\qq|$.

\smallskip

To use the machinery of forcing developed in \cref{subsec:forcing_with_closed_sets} we need to associate with each condition~$\pp\in \PP_1$ a closed set~$X^\pp$.

\begin{lemma} \label{lem:one_step:acceptability_of_closed_sets}
	The assignment of closed sets $X^\pp = [T^\pp]\setminus [B^\pp]^\prec = [T^\pp\setminus B^\pp]$ for $\pp\in \PP_1$ is acceptable (\cref{def:forcing_with_closed_sets}).
\end{lemma}

\begin{proof}
	Requirement~(b), that $X^\qq\subseteq X^\pp$ if~$\qq$ extends~$\pp$, follows directly from the definition of the partial ordering on~$\PP_1$.

	\smallskip

Let $\pp\in \PP_1$. Suppose that $[T^\pp]\subseteq [B^\pp]^\prec$. Since~$T^\pp$ is bounded, $[T^\pp]$ is compact.
This implies that there is a prefix-free, finite set $C\subset B^\pp$ such that every $\tau\in T^\pp$ is comparable with some element of~$C$. The collection of strings in~$T^\pp$ extended by some string in~$C$ witnesses that~$B^\pp$ is $h^\pp$-big above~$\s^\pp$. Since~$h^\pp\ge b^\pp$ above~$|\s^\pp|$ this implies that~$B^\pp$ is~$b^\pp$-big above~$\s^\pp$. We get requirement~(a): $X^\pp$ is nonempty.

\smallskip

Again let $\pp\in \PP_1$.
Let~$m \ge |\s^\pp|$. There is some~$\tau\in T^\pp$ of length~$m$ above which~$B^\pp$ is~$b^\pp$-small; otherwise, the concatenation property implies that~$B^\pp$ is $b^\pp$-big above~$\s^\pp$. If~$B^\pp$ is~$b^\pp$-small above~$\tau$ then
 	$   \qq = (\tau,T^\pp \cap \tau^{\preceq},B^\pp\cap \tau^\preceq,h^\pp,b^\pp)   $
is a condition in~$\PP_1$ extending~$\pp$ and satisfying $X^\qq\subseteq [T^\qq] \subseteq [\tau]^\prec$. This gives requirement~(c) of \cref{def:forcing_with_closed_sets}.
\end{proof}

As discussed in \cref{subsec:forcing_with_closed_sets}, if~$\GG\subset \PP_1$ is sufficiently generic then $\bigcap_{\pp\in \GG} [T^\pp\setminus B^\pp]$ is a singleton $\{x^\GG\}$. In fact
\[ x^\GG = \bigcup \left\{ \s^\pp  \,:\, \pp\in \GG \right\}.\]

Let~$\pp\in \PP_1$; since $B_\DNC \cap T^\pp\subseteq B^\pp$ we see that $X^\pp\subseteq \DNC$. Since strong forcing implies forcing (\cref{cor:forcing_and_generics}(3)) we get:

\begin{proposition} \label{prop:OneStep:DNC}
	Every condition in~$\PP_1$ forces that $x^\GG\in \DNC$.
\end{proposition}

\begin{remark} \label{rmk:OneStep:homogenising_conditions}
	Let~$A$ be an open set of strings and let~$g$ be a bounding function. We say that~$A$ is \emph{$g$-closed} if every string above which~$A$ is $g$-big is an element of~$A$.

	The concatenation property implies that every set~$A$ has a $g$-closure: the set of all strings above which~$A$ is $g$-big is $g$-closed.

	Let~$\pp\in \PP_1$. We could require that~$B^\pp$ be $b^\pp$-closed by replacing it by its $b^\pp$-closure. In this case $T^\pp\setminus B^\pp$ is an $(h^\pp-b^\pp)$-bushy tree with no leaves.

	In later sections we will use notions of largeness for which the concatenation property fails, and so will not be able to quite mimic this operation. Some amount of closure would be required to ensure that we get a restriction map from~$\PP_n$ to~$\PP_{n-1}$.
\end{remark}

\subsection{Totality} \label{subsec:step_one:totality}

Recall that for a set of strings~$C$ we let~$[C]^{\prec} = \bigcup_{\s\in C}[\s]^\prec$ be the set of $x\in \w^\w$ which extend some string in~$C$.

\begin{lemma} \label{lem:one_step:forcing_Sigma_1}
	Let $\pp\in \PP_1$. Let~$C\subseteq T^\pp$ be a c.e.\ and open in~$T^\pp$. Suppose that $\pp\force x^\GG\in [C]^\prec$. Let $\tau\in T^\pp$; let $g\in \Quick$ such that $h^\pp\gg g$, and $h^\pp\ge g \ge b^\pp$ above~$|\tau|$. Then the set $B^\pp\cup C$ is $g$-big above~$\tau$.
\end{lemma}

\begin{proof}
	Otherwise $  \qq = (\tau,T^\pp\cap \tau^\preceq,(B^\pp\cup C)\cap \tau^\preceq,h^\pp,g)$
	is a condition extending~$\pp$ which strongly forces that~$x^\GG\notin [C]^\prec$. (We need $g\ge b^\pp$ above~$|\tau|$ not to ensure that~$\qq$ is a condtion but to ensure that it extends~$\pp$.)
\end{proof}

\begin{remark} \label{rmk:one_step:bushiness_of_convergence}
	Let $\pp\in \PP_1$, let $C\subseteq T^\pp$ be c.e.\ and open in~$T^\pp$, and suppose that~$\pp$ strongly forces that~$x^\GG\in [C]^\prec$. By compactness there is some level~$m$ such that all strings in $T^\pp$ of length~$m$ are in $B^\pp\cup C$. This shows that $B^\pp\cup C$ is $h^\pp$-big above every~$\tau\in T^\pp$.
\end{remark}

The following proposition shows that we can always strongly force totality of~$\Gamma(x^\GG)$ for any Turing functional~$\Gamma$. Indeed it is equivalent to forcing totality, since every~$\Pi^0_2$ class is the domain of some Turing functional.

\begin{proposition} \label{prop:one_step:forcing_Pi_2}
	Let~$C\subseteq \w^\w$ be~$\Pi^0_2$ and let~$\pp\in \PP_1$. Then $\pp\force x^\GG\in C$ if and only if the set of conditions which strongly force that~$x^\GG\in C$ is dense below~$\pp$.
\end{proposition}

\begin{proof}
It suffices to show that if $\pp\force x^\GG\in C$ then~$\pp$ has an extension which strongly forces that~$x^\GG\in C$. Fix such~$\pp$.
	
	By \cref{lem:density_of_gg}, find some~$g\in \Quick$ such that $h^\pp\gg g\gg b^\pp$. As discussed above, every level of~$T^\pp$ contains a string above which~$B^\pp\cap T^\pp$ is $b^\pp$-small. So by extending~$\s^\pp$ (and taking the full subtree above that string) we may assume that $h^\pp\ge g\ge b^\pp$ above $|\s^\pp|$.

		Let~$\seq{C_k}$ be a uniform sequence of c.e.\ subsets of~$T^\pp$, open in~$T^\pp$, such that $C\cap [T^\pp] = [T^\pp]\cap \bigcap_k [C_k]^\prec$. \Cref{lem:one_step:forcing_Sigma_1} says that for all~$\tau\in T^\pp$, for all~$k$, the set $B^\pp\cup C_k$ is $g$-big above~$\tau$.
	
	\medskip
	
	We effectively define an increasing sequence $\seq{S_k}$ of finite $g$-bushy trees with the following properties:
\begin{itemize}
	\item $S_k$ is~$g$-bushy;
	\item $S_{k+1}$ is an end-extension of~$S_k$, and no leaf of~$S_k$ is a leaf of~$S_{k+1}$;
	\item $S_k\subset T^\pp$; and
	\item the leaves of $S_{k+1}$ lie in $B^\pp\cup C_k$.
\end{itemize}

	We start with $S_0 = \{\s^\pp\}$. We know that $B^\pp\cap C_0$ is $g$-big above~$\s^\pp$;
\Cref{lem:one_step:extended_concatenation} shows that for all $k>0$, $B^\pp\cup C_{k}$ is $g$-big above $B^\pp\cup C_{k-1}$. Thus, given~$S_k$ we can find a $g$-bushy end-extension~$S'_k$ of~$S_k$ with leaves in $B^\pp\cup C_k$; \cref{rmk:canonical_bushy_witness_is_a_subtree} shows that $S'_k\subset T^\pp$. Since~$T^\pp$ has no leaves, we can extend~$S'_k$ to the required $S_{k+1}$ by adding children from~$T^\pp$ to each leaf of~$S'_k$ (using the fact that $h^\pp\ge g$ above $|\s^\pp|$).

	Having defined the trees~$S_k$ we let $S = \bigcup_k S_k$. Then $S\subseteq T^\pp$, $S$ is $g$-bushy, and~$S$ has no leaves. Also,~$S$ is computable: a string of length~$k$ is in~$S$ if and only if it is in~$S_k$.

	Every path in~$S$ lies in $[B^\pp\cup C_k]^\preceq$ for all~$k$ and so $[S\setminus B^\pp]\subseteq C$. We required that $g\gg b^\pp$, so $\qq= (\s^\pp, S, B^\pp\cap S, g, b^\pp)$ is a condition which extends~$\pp$ and strongly forces that~$x^\GG\in C$.
\end{proof}

\subsection{Minimality}

We prove:

\begin{proposition} \label{prop:one_step:minimality}
	Every condition in~$\PP_1$ forces that $\deg_\Tur(x^\GG)$ is minimal.
\end{proposition}

Let~$\Gamma\colon \w^\w\to 2^\w$ be a Turing functional. There are three ways to ensure that~$\Gamma(x^\GG)$ does not violate the minimality of~$\deg_{\Tur}(x^\GG)$: ensuring that it is partial, ensuring that it is computable, or ensuring that it computes~$x^\GG$.

\medskip

For the rest of this section, fix a Turing functional~$\Gamma\colon \w^\w\to 2^\w$.

\begin{definition} \label{def:step_one:big_splittings}
	Let~$B$ be a set of strings. Two sets~$A_0$ and~$A_1$ of strings \emph{$\Gamma$-split mod~$B$} if $\Gamma(\tau_0)\perp \Gamma(\tau_1)$ for all $\tau_0\in A_0\setminus B$ and $\tau_1\in A_1\setminus B$.
\end{definition}

\begin{lemma} \label{lem:one_step:finding_splittings}
	Suppose that $\pp\in \PP_1$ strongly forces that $\Gamma(x^\GG)$ is total, and forces that~$\Gamma(x^\GG)$ is noncomputable.
	
	Let~$\tau\in T^\pp$. Let $g\in \Quick$ such that $h^\pp\gg g$, and $h^\pp\ge 3g$ and $g\ge b^\pp$ above~$|\tau|$. Then there are~$A_0,A_1\subset T^\pp$, each~$g$-big above~$\tau$, which $\Gamma$-split mod $B^\pp$.
\end{lemma}

\begin{proof}
	Suppose that~$\tau$ and~$g$ witness the failure of the lemma; we find an extension of~$\pp$ which forces that~$\Gamma(x^\GG)$ is computable.
	
	For $\alpha\in 2^{<\w}$ let
	\[ A_\alpha = B^\pp \cup \left\{ \rho\in T^\pp  \,:\,  \Gamma(\rho)\succeq \alpha \right\} \]
	and
	\[ A_{\perp \alpha} = \bigcup A_\beta\,\Cyl{\beta\in 2^{<\w} \andd \beta\perp \alpha} .\]

	Let~$\alpha\in 2^{<\w}$ and suppose that~$A_\alpha$ is $2g$-big above~$\tau$. By \cref{rmk:one_step:bushiness_of_convergence} the set $A_{\alpha\conc 0}\cup A_{\alpha\conc 1}$ is $h^\pp$-big above every $\rho\in A_{\alpha}$. Since $h^\pp \ge 2g$ above $|\tau|$, the concatenation property implies that $A_{\alpha\conc 0}\cup A_{\alpha\conc 1}$ is $2g$-big above~$\tau$. By the big subset property there is some $i<2$ such that $A_{\alpha\conc i}$ is $g$-big above~$\tau$ [Here we use that the range of~$\Gamma$ is in Cantor rather than Baire space; we also use this in the proof of \cref{lemma:one_step:combinatorial}].
	
	The assumption implies that $A_{\perp \alpha\conc i}$ is $g$-small above~$\tau$. Since $A_{\alpha\conc i}\cup A_{\perp \alpha\conc i}$ is $h^\pp$-big above~$\tau$ and $3g \le h^\pp$ above $|\tau|$ it must be that in fact~$A_{\alpha\conc i}$ is $2g$-big above~$\tau$.
	
	\medskip
	
	By recursion define the unique $z\in 2^\w$ such that for all $\alpha\prec z$, $A_\alpha$ is $2g$-big above~$\tau$. Note that~$z$ is computable. The set
	\[ A_{\perp z} =  \bigcup_{k<\w} A_{\perp z\,\,\rest k} \]
 is $g$-small above~$\tau$ because it is the union of an increasing sequence of sets, each $g$-small above~$\tau$; since largeness is witnessed by a finite tree, $g$-smallness above~$\tau$ is preserved when taking the union. The fact that~$z$ is computable shows that~$A_{\perp z}$ is c.e., whence the tuple $(\tau,T^\pp\cap \tau^\preceq,A_{\perp z}\cap \tau^\preceq, h^\pp,g)$
is a condition extending~$\pp$ as required (recalling that $B^\pp\subseteq A_{\perp z}$).
\end{proof}	
	
The following lemma will allow us to construct a ``delayed splitting'' subtree of~$T^\pp$.

\begin{lemma} \label{lem:one_step:finding_many_splittings}
		Suppose that $\pp\in \PP_1$ strongly forces that $\Gamma(x^\GG)$ is total, and forces that~$\Gamma(x^\GG)$ is noncomputable.
			Suppose that~$\tau_1,\tau_2,\dots, \tau_k\in T^\pp$.
Let $g\in \Quick$ such that $h^\pp\gg g$, and $g\ge b^\pp$ and $h^\pp\ge 3^kg$ above~$\min \{|\tau_1|,|\tau_2|, \dots, |\tau_k|\}$.
	Then there are sets $A_1,A_2,\dots, A_k\subset T^\pp$, each~$A_j$ $g$-big above~$\tau_j$, which pairwise $\Gamma$-split mod~$B^\pp$.
\end{lemma}

To prove \cref{lem:one_step:finding_many_splittings} we need the following, which (mod~$B$) is Lemma~6.2 of~\cite{KumabeLewis}.

\begin{lemma} \label{lemma:one_step:combinatorial}
	Let~$g,h\in \Quick$; let~$B$ be a set of strings. Suppose that:
	\begin{itemize}
		\item $\tau$ and $\tau^*$ are strings;
		\item $A$ is a set of strings, $3g$-big above~$\tau$;
		\item For all $\rho\in A$, $E_{\rho,0}$ and $E_{\rho,1}$ are $3g$-big above~$\rho$ and $\Gamma$-split mod~$B$;
		and
		\item $F$ is a set of strings, $3h$-big above~$\tau^*$, satisfying $|\Gamma(\s)|> |\Gamma(\nu)|$ for all $\s\in F\setminus B$ and all $\nu\in E\setminus B$, where $ E= \bigcup E_{\rho,i}\,\Cyl{\rho\in A, i<2}$.
	\end{itemize}
	Then there are $E'\subseteq E$, $g$-big above~$\tau$, and $F'\subseteq F$, $h$-big above~$\tau^*$, which $\Gamma$-split mod~$B$.
\end{lemma}

We delay the proof of \cref{lemma:one_step:combinatorial} until the end of the section.

\begin{proof}[Proof of \cref{lem:one_step:finding_many_splittings}, given \cref{lemma:one_step:combinatorial}]
	The proof is by induction on~$k$. The lemma is vacuous for~$k=1$. Assume the lemma has been proven for~$k$. Let~$\tau_1,\dots,\tau_k$ and $\tau^*$ be strings on $T^\pp$; suppose that $h^\pp\gg g$, and $h^\pp\ge 3^{k+1}g$ and $g\ge b^\pp$ above $\min \{|\tau^*|,|\tau_1|,|\tau_2|, \dots, |\tau_k|\}$. The hypothesis for~$k$ holds for the bound~$3g$ instead of~$g$, and so by induction we find finite sets~$A_1,\dots, A_k\subset T^\pp$, each~$A_j$ $3g$-big above~$\tau_j$, which pairwise~$\Gamma$-split mod~$B^\pp$. As per \cref{rmk:one_step:we_assume_that_big_sets_are_leaves_of_trees} we assume that $A_j\subset \tau_j^\preceq$.
	
	For every $j=1,\dots, k$, for every $\rho\in A_j$, by \cref{lem:one_step:finding_splittings} find finite $E_{\rho,0}$ and~$E_{\rho,1}$, subsets of~$T^\pp$, each~$3g$-big above~$\rho$ and contained in~$\rho^\preceq$, which $\Gamma$-split mod~$B^\pp$. Let $E_j = \bigcup E_{\rho,i}\Cyl{\rho\in A_j,i<2}$. Note that the~$E_j$ also pairwise $\Gamma$-split mod~$B^\pp$.

	Since~$\bigcup_{j\le k}E_j$ is finite,~$\pp$ strongly forces totality of~$\Gamma(x^\GG)$ and $3^{k+1}g \le h^\pp$ above $|\tau^*|$, by \cref{rmk:one_step:bushiness_of_convergence} we find $F\subset T^\pp$ which is $3^{k}g$-big above~$\tau^*$, such that $|\Gamma(\s)|> |\Gamma(\nu)|$  for all $\s\in F\setminus B^\pp$ and $\nu\in \bigcup_{j\le k} E_j\setminus B^\pp$.
	
	\medskip

	Let $F_{k} = F$. By (reverse) recursion on $j = k,k-1,\dots, 1$ we define sets $E'_j\subseteq E_j$ and $F_{j-1}\subseteq F_{j}$ such that every $E'_j$ is $g$-big above~$\tau_j$, $F_j$ is $3^{j}g$-big above~$\tau^*$ and $E'_j$ and $F_{j-1}$ pairwise $\Gamma$-split mod~$B^\pp$. To do this, given~$F_{j}$ apply \cref{lemma:one_step:combinatorial} with $\tau=\tau_j$, $A = A_j$, $g$, $\tau^*$ and $E_{\rho,i}$ as themselves, $F = F_{j}$ and $h = 3^{j-1}g$.
	
	In the end the sets $E'_j$ for $j\le k$ and~$F_0$ are as required. 	
\end{proof}

\begin{proposition} \label{prop:one_step:splitting_tree}
	Every condition in~$\PP_1$ forces that if~$\Gamma(x^\GG)$ is total and noncomputable then $\Gamma(x^\GG)\equiv_\Tur x^\GG$.
\end{proposition}

\begin{proof}
It suffices to show that if $\pp\in \PP_1$ forces that~$\Gamma(x^\GG)$ is total and noncomputable then~$\pp$ has an extension which forces that $\Gamma(x^\GG)\equiv_\Tur x^\GG$. By \cref{prop:one_step:forcing_Pi_2} we may assume that~$\pp$ strongly forces that~$\Gamma(x^\GG)$ is total.
	
By \cref{lem:density_of_gg} find some~$g\in \Quick$ such that $h^\pp\gg g \gg b^\pp$. Let~$\bar g(n) = \prod_{m<n} g(m)$. As above by extending~$\s^\pp$ we may assume that $h^\pp \ge 3^{\bar g}g$ and $g\ge b^\pp$ above~$|\s^\pp|$ (see \cref{lem:bounded_products}).
	
We effectively define an increasing sequence~$\seq{\ell_{k}}$ and a sequence $\seq{S_k}$ of finite subtrees of~$T^\pp$ such that: (a) $S_{k+1}$ is an end-extension of~$S_k$; (b) the leaves of~$S_k$ all have length~$\ell_k$; and (c)~$S_k$ is \emph{exactly}~$g$-bushy: every nonterminal $\tau\in S_k$ has precisely $g(|\tau|)$ many immediate extensions on~$S_k$.
		
	Let $\ell_{0} = |\s^\pp|$ and $S_0 = \{\s^\pp\}$. Suppose that~$S_k$ and~$\ell_{k}$ have been defined. For every leaf~$\tau$ of~$S_k$ we find a finite tree $R_\tau\subset T^\pp$, exactly $g$-bushy above~$\tau$, such that the sets of leaves of the various~$R_\tau$ pairwise $\Gamma$-split mod~$B^\pp$. This can be done since the number of leaves of~$S_k$ is $\prod_{m\in [|\s^\pp|,\ell_{k})} g(m)$, which is bounded by~$\bar g(\ell_{k})$. We assumed that $h^\pp\ge 3^{\bar g}g$ and so $h^\pp\ge 3^{\bar g(\ell_k)}g$ above $\ell_k$; so \cref{lem:one_step:finding_many_splittings} applies.
	
	Let~$S'_k$ be the union of~$S_k$ with the trees $R_\tau$ for all leaves~$\tau$ of~$S_k$. Let~$\ell_{k+1}$ be greater than the height of~$S'_k$; obtain~$S_{k+1}$ by appending a subtree of~$T^\pp$, exactly $g$-bushy above~$\rho$, to every leaf~$\rho$ of $S'_k$.

	\smallskip

	Let $S = \bigcup_k S_k$. As in the proof of \cref{prop:one_step:forcing_Pi_2}, $S$ is computable, computably bounded and has no leaves. It is~$g$-bushy, and~$\Gamma$ is 1-1 on $[S\setminus B^\pp]$: if $x,x'\in [S\setminus B^\pp]$ and $x\rest{\ell_k}\ne x'\rest{\ell_k}$ then $\Gamma(x\rest{\ell_{k+1}})\perp \Gamma(x'\rest{\ell_{k+1}})$. The tuple $(\s^\pp,S,B^\pp\cap S, g, b^\pp)$ is a condition as required (\cref{lem:1-1_or_constant_on_compact_space}).
\end{proof}

\begin{proof}[Proof of \cref{prop:one_step:minimality}]
	Let~$\pp\in \PP_1$. Let~$\Gamma$ be a Turing functional. If~$\pp$ has an extension which forces that $\Gamma(x^\GG)$ is partial then we are done. Otherwise~$\pp$ forces that~$\Gamma(x^\GG)$ is total. We can extend~$\pp$ to a condition~$\qq$ which decides whether~$\Gamma(x^\GG)$ is computable or not. If the former then we are done. Otherwise, \cref{prop:one_step:splitting_tree} says that~$\qq$ forces that~$\Gamma(x^\GG)\equiv_\Tur x^\GG$.
\end{proof}

\begin{proof}[Proof of \cref{lemma:one_step:combinatorial}]
	Let $E = \bigcup E_{\rho,i}\,\Cyl{i<2 \andd \rho\in A}$.
	
	For a string $\alpha\in 2^{<\w}$ let
	\[ F_{\succeq \alpha} = (F\cap B) \cup \left\{ \s\in F  \,:\, \Gamma(\s)\succeq \alpha \right\} ,\]
	and similarly define $F_{\perp\alpha}$, $E_{\succeq\alpha}$, $E_{\preceq\alpha}$ and so on.
	
	\smallskip
	
	If $F\cap B$ is $h$-big above~$\tau^*$ then we can let~$F' = F\cap B$ and $E' = E$. Similarly if $E\cap B$ is~$g$-big above~$\tau$.
	
	 Suppose otherwise. In that case, for sufficiently long~$\alpha$, $F_{\succeq \alpha}$ is $h$-small above~$\tau^*$ (as it equals $F\cap B)$. Let~$\alpha$ be a string, maximal with respect to $F_{\succeq \alpha}$ being~$h$-big above~$\tau^*$. We will show that either
	\begin{enumerate}
		\item $E_{\perp\alpha}$ is $g$-big above~$\tau$, or
		\item $E_{\succeq \alpha}$ is $g$-big above~$\tau$ and $F_{\perp\alpha}$ is $h$-big above~$\tau^*$.
	\end{enumerate}
	In both cases we can find~$E'$ and~$F'$ as required.
	
	We examine two cases, depending on $E_{\preceq \alpha}$.
	
	\smallskip
	
	First, suppose that~$E_{\preceq \alpha}$ is $g$-big above~$\tau$. Let~$R$ be a tree witnessing this. Every leaf of~$R$ extends some element of~$A$, so every element of~$R$ is comparable with some element of~$A$. Since~$A$ is an antichain, the restriction of~$R$ to initial segments of elements of~$A$ is $g$-bushy. This shows that~$A'$, the set of $\rho\in A$ such that $E_{\preceq \alpha}$ is $g$-big above~$\rho$, is $g$-big above~$\tau$.  We show that $E_{\perp\alpha}$ is~$g$-big above every~$\rho\in A'$; with the concatenation property this implies~(1). Let~$\rho\in A'$; there are two possibilities. If~$B\cap E$ is $g$-big above~$\rho$ then we are done. Otherwise for some~$i<2$, $E_{\rho,i}$ intersects $E_{\preceq \alpha}\setminus B$. But then $E_{\rho,1-i}\subseteq E_{\perp\alpha}$, and $E_{\rho,1-i}$ is $3g$-big above~$\rho$.
	
\smallskip
	
	In the second case, suppose that~$E_{\preceq\alpha}$ is $g$-small above~$\tau$. Since~$E = E_{\perp\alpha}\cup E_{\succeq \alpha}\cup E_{\preceq \alpha}$ is~$3g$-big above~$\tau$, either (1) holds, or~$E_{\succeq \alpha}$ is $g$-big above~$\tau$. Assume the latter. We assumed that $E\cap B$ is $g$-small above~$\tau$; together, we see that $E_{\succeq \alpha}\setminus B$ is nonempty. In turn this implies that $|\Gamma(\s)|>|\alpha|$ for all~$\s\in F\setminus B$; so $F = F_{\succneq \alpha}\cup F_{\perp\alpha}$.
	
	The maximality of~$\alpha$ ensures that $F_{\succneq \alpha}$ is $2h$-small above~$\tau^*$ [Here again we use the fact that~$\Gamma$ maps into Cantor space]. Since~$F$ is $3h$-big above~$\tau^*$ it must be that $F_{\perp\alpha}$ is $h$-big above~$\tau^*$, so (2) holds. 	
\end{proof}

\section{A relative DNC strong minimal cover of a DNC minimal degree} \label{sec:step_two}

We now construct two sequences $x,y\in \w^\w$ such that $x\in \DNC$, $x$ has minimal Turing degree, $y\in \DNC^x$ and $\deg_\Tur(x,y)$ is a strong minimal cover of~$\deg_\Tur(x)$.

We use the mechanism of tree systems that was used by Cai~\cite{Cai:2minimalANR,Cai:2-minimal_nonGL2,Cai:2-minimal_GH1} to show that there is a generalised high degree which is a minimal cover of a minimal degree. This is a more versatile approach than the homogenous trees which are usually used to construct initial segments of the Turing degrees (as in~\cite{Lerman:book}).

\subsection{Length 2 tree systems} \label{subsec:step_two:tree_systems}

Let~$A\subseteq \w^{<\w}\times \w^{<\w}$ be a set of pairs of strings. For $\tau\in \w^{<\w}$ we let
\[ A(\tau)  = \left\{ \rho\in \w^{<\w} \,:\,  (\tau,\rho)\in A \right\}. \]
Of course $\dom A = \left\{ \tau \,:\,  \exists \rho\,\,[(\tau,\rho)\in A] \right\}$.

\begin{definition} \label{def:TwoStep:TreeSystem}
	A \emph{tree system} of length 2 above a pair~$(\s,\mu)$ is a set~$T$ of pairs of strings satisfying:
	\begin{itemize}
		\item $\dom T$ is a tree above~$\s$;
		\item For all $\tau\in \dom T$, $T(\tau)$ is a finite tree above~$\mu$; and
		\item If $\tau\prec \tau'$ are in~$\dom T$ then $T(\tau')$ is an end-extension of $T(\tau)$.
	\end{itemize}
\end{definition}

In this section we only consider systems of length 2 and so we omit mentioning the length.

\smallskip

A tree system~$S$ is a subsystem of~$T$ if $S\subseteq T$. This means that $\dom S$ is a subtree of $\dom T$ and for all~$\tau\in \dom S$, $S(\tau)$ is a subtree of~$T(\tau)$. If $(\tau,\rho)\in T$ then $T\cap (\tau,\rho)^\preceq$ is a tree system, the system whose domain is the full subtree of~$\dom T$ above~$\tau$ and which maps all~$\tau'$ in its domain to the full subtree of~$T(\tau')$ above~$\rho$. Here of course $(\tau,\rho)^\preceq = \tau^\preceq \times \rho^\preceq$ is the upwards-closure of $\{(\tau,\rho)\}$ in the partial ordering $\preceq$ on $\Strings{2}$ defined by the product of extension on strings: $(\tau,\rho)\preceq (\tau',\rho')$ if $\tau\preceq \tau'$ and $\rho\preceq \rho'$.

\smallskip

A tree system is $h$-\emph{bounded} if for all~$(\tau,\rho)\in T$, $\tau(n) < h(n)$ for all $n< |\tau|$ and $\rho(n)< h(n)$ for all $n<|\rho|$. It is \emph{computably bounded} if it is bounded by some computable function.

If~$T$ is a computable and computably bounded tree system then $\dom T$ is computable and the map $\tau\mapsto T(\tau)$ is computable (for each $\tau\in \dom T$ we obtain a canonical index for $T(\tau)$ as a finite set).

\subsubsection*{Forest systems} 
\label{ssub:forest_systems}

To iterate largeness we require the notion of forest systems.

We call a set of pairs of strings $A\subset \Strings{2}$ \emph{prefix-free} if $\dom A$ is prefix-free and for all $\tau\in \dom A$, $A(\tau)$ is prefix-free. For a set of pairs~$A$ let $A^\preceq = \bigcup_{(\s,\mu)\in A} (\s,\mu)^\preceq$ be the upwards closure of~$A$ under~$\preceq$. If~$A$ is prefix-free then $A^\preceq$ is the \emph{disjoint} union of $(\s,\mu)^\preceq$ for $(\s,\mu)\in A$. In other words, if $(\tau,\rho)$ extends some element of~$A$ then that element is unique. We denote this element by $(\tau,\rho)^{-A}$.

\begin{definition} \label{def:TwoStep:ForestSystem}
	A \emph{forest system} of length 2 above a finite prefix-free set~$A\subset \Strings{2}$ is a set~$T$ of pairs of strings satisfying:
	\begin{itemize}
		\item $\dom T$ is a forest above~$\dom A$;
		\item For all $\tau\in \dom T$, $T(\tau)$ is a finite forest above~$A(\tau^{-\dom A})$ (where again $\tau^{-\dom A}$ is $\tau$'s unique predecessor in~$\dom A$); and
		\item If $\tau\prec \tau'$ are in~$\dom T$ then $T(\tau')$ is an end-extension of $T(\tau)$.
	\end{itemize}
\end{definition}

A \emph{leaf} of a forest system~$T$ is a pair~$(\tau,\rho)\in T$ such that $\tau$ is a leaf of~$\dom T$ and~$\rho$ is a leaf of~$T(\tau)$. Equivalently, it is a maximal element of the set of pairs~$T$, if~$T$ is partially ordered by double extension~$\preceq$. The set of leaves of a finite forest system is prefix-free.


\subsubsection*{Paths of tree systems}

Let~$T$ be a tree system above $(\s,\mu)$. For $x\in [\dom T]$ we let
\[ T(x) = \bigcup T(\tau)\Cyl{\s\preceq \tau\prec x}. \]
We also let
 \[ [T] = \{ (x,y)\,:\, x\in [\dom T]\andd y\in [T(x)]\} .\]
In general the set~$[T]$ need not be closed.

\begin{lemma} \label{lem:TwoSteps:Set_of_paths_is_closed}
	Suppose that for all~$x\in [\dom T]$ the tree $T(x)$ has no leaves. Then~$[T]$ is a closed subset of $[\s,\mu]^\prec$.
\end{lemma}

\begin{proof}
	For~$\tau\in \dom T$ let
	\[ E_\tau = \bigcup \,[\rho]^\prec \,\Cyl{\rho\text{ a leaf of }T(\tau)}; \]
	for $n\ge |\s|$ let
	\[ E_n = \bigcup \, \left([\tau]^\prec\times E_\tau\right) \,\Cyl{\tau\in \dom T \andd |\tau|=n} .\]
	Each~$E_n$ is clopen. We show that $[T]=\bigcap E_n$. We always have $[T]\subseteq \bigcap_{n\ge |\s|}E_n$. For suppose that~$(x,y)\in [T]$, and let $n\ge |\s|$. Let $\tau = x\rest n$; so $\tau\in \dom T$. Let~$m$ be greater than the height of~$T(\tau)$, and let $\rho = y\rest{m}$. Since~$\rho\in T(x)$ there is some~$\tau'\prec x$ such that $\rho\in T(\tau')$. Since $\rho\notin T(\tau)$ we must have $\tau\prec \tau'$, and so~$\rho$ extends some leaf of~$T(\tau)$; this shows that $y\in E_\tau$, so $(x,y)\in E_n$.
	
	In the other direction we use our assumption. Suppose that $(x,y)\in \bigcap_{n\ge |\s|} E_n$. For all~$n\ge|\s|$, $(x,y)\in E_n$ implies that $x\rest n\in \dom T$, so $x\in [\dom T]$. For all~$n\ge |\s|$, some leaf of~$T(x\rest{n})$ is an initial segment of~$y$. To show that $y\in [T(x)]$ it suffices to show that the minimum length of a leaf in $T(x\rest{n})$ is unbounded as $n\to \infty$. But otherwise~$T(x)$ would have a leaf.
\end{proof}

We will require that the pairs in tree systems appearing in our conditions can be extended to paths. It is not enough to require that the system not have leaves.

\begin{lemma} \label{lem:TwoSteps:always_extendible_is_local_by_compactness}
	 Let~$T$ be a bounded tree system and suppose that~$\dom T$ has no leaves. The following are equivalent:
	\begin{enumerate}
			\item For all $k$ there is some~$m$ such that for every $\tau\in \dom T$ of length~$m$, every leaf of $T(\tau)$ has length at least~$k$.
		\item For all $x\in [\dom T]$, $T(x)$ has no leaves.
	\end{enumerate}
\end{lemma}

\begin{proof}
	That~(1) implies~(2) is immediate. Suppose~(2) holds. By \cref{lem:TwoSteps:Set_of_paths_is_closed}, $[T]$ is closed; since~$T$ is bounded, $[T]$ is compact. Let~$k<\w$. The collection of clopen rectangles $[\tau,\rho]^\prec$ where $\tau\in \dom T$, $\rho$ is a leaf of~$T(\tau)$, and $|\rho|\ge k$ is an open cover of~$[T]$; a finite sub-cover gives the desired~$m$. 	
\end{proof}

To simplfy the combinatorics of finding big splittings, we restrict ourselves to ``balanced'' tree systems.

\begin{definition} \label{def:two_step:balanced_tree_systems}
	Let~$T$ be a tree system and let~$n<\w$. We say that~$m$ is a \emph{balanced level of~$T$} if for all $\tau\in \dom T$ of length~$m$, every leaf of~$T(\tau)$ has length~$m$. We say that~$T$ is \emph{balanced} if $\dom T$ has no leaves and~$T$ has infinitely many balanced levels.
\end{definition}

If~$T$ is bounded and balanced then it satisfies the conditions of \cref{lem:TwoSteps:always_extendible_is_local_by_compactness} and so by \cref{lem:TwoSteps:Set_of_paths_is_closed}, $[T]$ is closed. If~$T$ is balanced, computable and computably bounded then~$[T]$ is effectively closed (this is really where we use the requirement that if~$\tau'$ extends~$\tau$ in~$\dom T$ then $T(\tau')$ is an end-extension, rather than any extension, of~$T(\tau)$).

\subsection{Bushiness for forest systems}

\begin{definition} \label{def:TwoSteps:bushy}
	Let~$g$ and~$h$ be bounding functions. A forest system~$T$ is \emph{$(g,h)$-bushy} if $\dom T$ is $g$-bushy and for all~$\tau\in \dom T$, $T(\tau)$ is $h$-bushy.
\end{definition}

\begin{lemma} \label{lem:two_step:bigness}
Let~$A\subset \Strings{2}$ be finite and prefix-free, and let~$g$ and~$h$ be bounding functions.
The following are equivalent for a set~$B$ of pairs of strings:
\begin{enumerate}
	\item There is finite $(g,h)$-bushy forest system above~$A$, all of whose leaves lie in~$B$.
	\item The set of~$\tau$ such that $B(\tau)$ is $h$-big above~$A(\tau^{-\dom A})$ is $g$-big above~$\dom A$.
\end{enumerate}
\end{lemma}

\begin{proof}
Assume~(2). We define a forest system~$S$ by first defining~$\dom S$, and then for all $\tau\in \dom S$, defining $S(\tau)$. We let~$\dom S$ be a $g$-bushy forest above~$\dom A$ such that for every leaf~$\tau$ of~$\dom S$, $B(\tau)$ is $h$-big above~$A(\tau^{-\dom A})$. Now let~$\tau\in \dom S$; let $\s = \tau^{-\dom A}$. There are two cases. If $\tau$ is a leaf of~$\dom S$ then we let $S(\tau)$ be an $h$-bushy forest above~$A(\s)$ which witnesses that $B(\tau)$ is $h$-big above~$A(\s)$. If~$\tau$ is not a leaf of~$\dom S$ then we let $S(\tau) = A(\s)$.
\end{proof}

These equivalent conditions define the notion of~$B$ being \emph{$(g,h)$-big} above~$A$; if they fail, we say that~$B$ is \emph{$(g,h)$-small} above~$A$. If~$A$ is infinite then we say that~$B$ is $(g,h)$-big above~$A$ if it is $(g,h)$-big above every finite prefix-free subset of~$A$.

For brevity we let for $B\subseteq \Strings{2}$, a bounding function~$h$ and a finite prefix-free set of strings~$D$
 \[ \project{h}{D}{B}  = \left\{ \tau  \,:\, B(\tau) \text{ is $h$-big above }D \right\} .\]
Note that $\project{h}{D}{B} = \bigcap_{\rho\in D} \project{h}{\rho}{B}$. A set $B$ is $(g,h)$-big above a finite prefix-free set~$A$ if and only if for all~$\s\in \dom A$, $\project{h}{A(\s)}{B}$ is $g$-big above~$\s$.

\smallskip

The big subset property holds.

\begin{lemma} \label{lem:TwoSteps:BigSubset}
	Let~$g,g'$ and~$h,h'$ be bounding functions and let~$(\s,\mu)\in \Strings{2}$. Suppose that~$B,C\subseteq \Strings{2}$ and that $B\cup C$ is $(g+g',h+h')$-big above~$(\s,\mu)$. Then either $B$ is $(g,h)$-big above~$(\s,\mu)$ or $C$ is $(g',h')$-big above~$(\s,\mu)$.
\end{lemma}

\begin{proof}
	The set $\project{h+h'}{\mu}{B\cup C}$ is $(g+g')$-big above~$\s$.
The big subset property implies that $\project{h+h'}{\mu}{B\cup C}\subseteq \project{h}{\mu}{B}\cup \project{h'}{\mu}{C}$. Utilising the big subset property again, this time on the left coordinate, we see that either $\project{h}{\mu}{B}$ is $g$-big above~$\tau$ or $\project{h'}{\mu}{C}$ is $g'$-big above~$\tau$. The first means that $B$ is $(g,h)$-big above $(\s,\mu)$; the second, that~$C$ is $(g',h')$-big above~$(\s,\mu)$.
\end{proof}

\subsubsection*{Weak concatenation} 
\label{subsub:the_weak_concatenation_property}

The concatenation property fails. Suppose that~$A$ is $(g,h)$-big above $(\s,\mu)$, and that~$B$ is $(g,h)$-big above every $(\tau,\rho)\in A$. It is possible that~$B$ is not $(g,h)$-big above $(\s,\mu)$: take for example two strings $\rho_1$ and $\rho_2$ and a string~$\tau$ such that $(\tau,\rho_1), (\tau,\rho_2)\in A$. Then $\project{h}{\rho_1}{B}$ and $\project{h}{\rho_2}{B}$ are both $g$-big above~$\tau$, but the trees witnessing these facts need not be the same. That is, it is possible that $\project{h}{\{\rho_1,\rho_2\}}{B}$ is $g$-small above~$\tau$. As a result, it is possible that a set~$B$ is $(g,h)$-small above some~$(\s,\mu)$ but the set of pairs above which~$B$ is $(g,h)$-big is $(g,h)$-big above~$(\s,\mu)$. Instead, we will employ a weak version of the concatenation property.

\begin{definition} \label{def:two_step:end_extension}
	Let~$S$ and~$R$ be forest systems. We say that~$R$ is an \emph{end-extension} of~$S$ if:
	\begin{itemize}
		\item $\dom R$ is an end-extension of $\dom S$;
		\item If $\tau\in \dom S$ is not a leaf of $\dom S$, then $R(\tau) = S(\tau)$;
		\item If~$\tau$ is a leaf of~$\dom S$ then $R(\tau)$ is an end-extension of $S(\tau)$.
	\end{itemize}
\end{definition}

Note that this relation is transitive. Now if~$T$ is a finite (length 1) forest above~$A$, $E$ is the set of leaves of~$T$, and~$U$ is a forest above~$E$, then $T\cup U$ is a forest above~$A$, an end-extension of~$T$ whose leaves are the leaves of~$U$. For forest systems we cannot take unions. Suppose that~$S$ is a finite forest system above~$A$; let~$D$ be the set of leaves of~$S$, and suppose that~$R$ is a forest system above~$D$. We define the concatenation $S\conc R$ of~$S$ and~$R$:
\begin{itemize}
	\item $\dom(S\conc R) = (\dom S)\cup (\dom R)$;
	\item For $\tau\in \dom S\setminus \dom D$, $(S\conc R)(\tau) = S(\tau)$;
	\item For $\tau\in \dom R$, $(S\conc R)(\tau) = (S(\tau^{-\dom D}))\cup R(\tau)$.
\end{itemize}
This is a forest system above~$A$, an end-extension of~$S$ whose leaves are the leaves of~$R$. Note that if~$\tau\in \dom D$ then we do not assume that $R(\tau) = D(\tau)$, and so it is possible that $(S\conc R)(\tau)\ne S(\tau)$. If both~$S$ and~$R$ are $(g,h)$-bushy then so is $S\conc R$. We conclude:

\begin{lemma} \label{lem:two_step:concatenating_forest_systems_and_bigness}
	Suppose that~$B$ is $(g,h)$-big above~$A$, and that~$C$ is $(g,h)$-big above~$B$. Then~$C$ is $(g,h)$-big above~$A$. Indeed, every finite $(g,h)$-bushy forest system whose leaves are in~$B$ has a finite $(g,h)$-bushy end-extension whose leaves are in~$C$.
\end{lemma}

A set~$B$ of pairs of strings is \emph{open} if it is upwards closed in the partial ordering~$\preceq$: closed under taking extensions in either coordinate. 

The following lemma concerns sets of \emph{strings}, not pairs of strings. It is a consequence of the concatenation property, and is formally proved by induction on $|\mathcal B|$.

\begin{lemma} \label{lem:two_step:one_step_induction_for_weak_concatenation}
	Let~$\mathcal B$ be a finite collection of open sets of strings, and let~$A$ be a finite, prefix-free set of strings. Suppose that each~$B\in \mathcal B$ is $g$-big above every $\s\in A^\preceq$. Then $\bigcap \mathcal B$ is $g$-big above~$A$. \qedhere
\end{lemma}

\begin{lemma} \label{lem:two_step:last_step_for_weak_concatenation}
	Let~$A$ and~$B$ be sets of pairs of strings, and let~$g$ and~$h$ be bounding functions. Suppose that~$B$ is open. Suppose that for all $(\s,\mu)\in A$, for all $\s'\succeq \s$, $B$ is $(g,h)$-big above $(\s',\mu)$. Then~$B$ is $(g,h)$-big above~$A$.
\end{lemma}

\begin{proof}
	 It suffices to show that for any $\s\in \dom A$ and any finite, prefix-free $E\subseteq A(\s)$, $\project{h}{E}{B}$ is $g$-big above~$\s$. We apply \cref{lem:two_step:one_step_induction_for_weak_concatenation} to the collection of sets $\project{h}{\mu}{B}$ for $\mu\in E$. The fact that~$B$ is open implies that each $\project{h}{\mu}{B}$ is open; the assumption is that each $\project{h}{\mu}{B}$ is $g$-big above every extension of~$\s$.
\end{proof}

\begin{corollary}[Weak concatenation property] \label{lem:two_step:weak_concatentation}
	Let~$A$, $B$ and~$C$ be sets of pairs of strings, and suppose that~$C$ is open. Suppose that~$B$ is $(g,h)$-big above~$A$, and that for all $(\tau,\rho)\in B$, for all $\tau'\succeq \tau$, $C$ is $(g,h)$-big above $(\tau',\rho)$. Then~$C$ is $(g,h)$-big above~$A$.
\end{corollary}

\subsubsection*{Working within tree systems} 
\label{ssub:working_within_tree_systems}

We will need to apply the weak concatenation property while working within a given tree system~$T$.

\begin{remark} \label{rmk:two_step:minimal_bigness_witness}
	Suppose that $B$ is $(g,h)$-big above $A$, that~$T$ is a tree system and that $A,B\subseteq T$. Then the forest system constructed in the proof of \cref{lem:two_step:bigness} is a subset of~$T$.
\end{remark}

Fix a tree system~$T$. Suppose that~$S$ is finite forest system; let~$D$ be the set of leaves of~$S$. Let~$R$ be a forest system above~$D$. Suppose that both~$S$ and~$R$ are subsets of~$T$. Then $S\conc R$ is also a subset of~$T$. Thus, \cref{rmk:two_step:minimal_bigness_witness} can be extended. Suppose that $B$ is $(g,h)$-big above~$A$, that~$C$ is $(g,h)$-big above~$B$, and that $A,B,C\subseteq T$. Then not only is there a finite $(g,h)$-bushy forest system $S\subseteq T$ above~$A$ whose leaves are in~$B$, but further, any such system~$S$ can be end-extended to a finite $(g,h)$-bushy forest system $R\subseteq T$ above~$A$ whose leaves are in~$C$.

\smallskip

 If~$T$ is a tree system and~$B\subseteq T$ then we say that~$B$ is \emph{open in~$T$} if it is upwards closed in the restriction of the partial ordering~$\preceq$ to~$T$. \Cref{lem:two_step:one_step_induction_for_weak_concatenation} can be ``restricted to a tree~$S$'': if $A,B\subseteq S$ and each~$B\in \mathcal B$ is open in~$S$ and $g$-big above $A^\preceq \cap S$, then $\bigcap\mathcal B$ is $g$-big above~$A$. We then obtain a version of \cref{lem:two_step:last_step_for_weak_concatenation} restricted to~$T$:

\begin{lemma} \label{lem:two_step:last_step_for_weak_concatenation_witin_T}
	Let~$T$ be a tree system; let~$A,B\subseteq T$, and let~$g$ and~$h$ be bounding functions. Suppose that~$B$ is open in~$T$, and that for all $(\s,\mu)\in A$, for all $\s'\succeq \s$ in $\dom T$, $B$ is $(g,h)$-big above $(\s',\rho)$. Then~$B$ is $(g,h)$-big above~$A$.
\end{lemma}

And so we get the weak concatenation property within a tree system:

\begin{corollary} \label{lem:two_step:restricted_weak_concatentation}
	Let~$T$ be a tree system, let~$A,B,C\subseteq T$, and suppose that~$C$ is open in~$T$. Suppose that~$B$ is $(g,h)$-big above~$A$, and that for all $(\tau,\rho)\in B$, for all $\tau'\succeq \tau$ in $\dom T$, $C$ is $(g,h)$-big above $(\tau',\rho)$. Then~$C$ is $(g,h)$-big above~$A$, and in fact every finite $(g,h)$-bushy forest system $S\subseteq T$ which witnesses that~$B$ is $(g,h)$-big above~$A$ has an end-extension $R\subseteq T$ which witnesses that~$C$ is $(g,h)$-big above~$A$.
\end{corollary}

We obtain a lemma which would allow us to take full subsystems as extensions.

\begin{lemma} \label{lem:two_step:can_take_full_subtree}
	Let~$T$ be a bounded and balanced $(b,c)$-bushy tree system above $(\s,\mu)$ and let $B\subset T$ be open in~$T$ and $(b,c)$-small above $(\s,\mu)$. Then for every~$m$ there is some $(\tau,\rho)\in T$ such that $|\tau|,|\rho|\ge m$ and above which~$B$ is $(b,c)$-small.
\end{lemma}

\begin{proof}
	Let~$m$ be some balanced level of~$T$. Let~$D$ be the set of pairs $(\tau,\rho)\in T$ such that $|\tau| = |\rho|\ge m$. Then~$D$ is $(b,c)$-big above $(\s,\mu)$. If there is no pair as required then the weak concatenation property localised to~$T$ (\cref{lem:two_step:restricted_weak_concatentation}) shows that $B$ is $(b,c)$-big above~$(\s,\mu)$.
\end{proof}

\begin{remark}  \label{rmk:two_step:we_assume_that_big_sets_are_leaves_of_tree_systems}
We use the same convention discussed in \cref{rmk:one_step:we_assume_that_big_sets_are_leaves_of_trees}; we assume that large sets given to us are sets of leaves of tree systems witnessing their largeness. For example, if we are given a set~$B$ of pairs, $(g,h)$-big above some~$A$, then we assume that~$B$ is finite and prefix-free; that for all~$\tau\in \dom B$, $B(\tau)$ is $h$-big above $A(\tau^{-\dom A})$; and that~$B\subseteq A^\preceq$.
\end{remark}

\subsection{The notion of forcing and the generic} \label{subsec:step_two:the_notion}

Let~$B_{\DNC^2}$ be the set of pairs~$(\tau,\rho)$ such that $\tau\in B_{\DNC}$ or $\rho\in B_{\DNC^\tau}$; the latter means that $\rho(e) = J^\tau(e)\converge$ for some $e<|\rho|$. Note that this set of pairs is $(2,2)$-small above $(\emptystring,\emptystring)$.

\smallskip

We let~$\PP_2$ be the set of tuples $\pp = ((\s^\pp,\mu^\pp),T^\pp,B^\pp,h^\pp,b^\pp)$ satisfying:
\begin{enumerate}
	\item $T^\pp$ is a computably bounded, computable, balanced tree system above~$(\s^\pp,\mu^\pp)$;
	\item $h^\pp\in \Quick$ and~$T^\pp$ is~$(h^\pp,h^\pp)$-bushy;
	\item $B^\pp\subset T^\pp$ is c.e.\ and open in~$T^\pp$, and $B^\pp\supseteq B_{\DNC^2}\cap T^\pp$;
	\item $b^\pp\in \Quick$ and $B^\pp$ is $(b^\pp,b^\pp)$-small above~$(\s^\pp,\mu^\pp)$; and
	\item $h^\pp\gg b^\pp$ and $h^\pp\ge b^\pp$ above $\min \{|\s^\pp|,|\mu^\pp|\}$.
\end{enumerate}

We define a partial ordering on~$\PP_2$ as follows. A condition~$\qq$ extends a condition~$\pp$ if $(\s^\pp,\mu^\pp)\preceq (\s^\qq,\mu^\qq)$, $T^\qq$ is a subsystem of~$T^\pp$, $B^\pp \cap T^\qq\subseteq B^\qq$, and $h^\qq\le h^\pp$ and $b^\qq\ge b^\pp$ above $\min\{ |\s^\qq|,|\mu^\qq|\}$.

\begin{lemma} \label{lem:two_step:acceptability_of_closed_sets}
	The assignment of closed sets $X^\pp = [T^\pp]\setminus [B^\pp]^\prec$ for $\pp\in \PP_2$ is acceptable (\cref{def:forcing_with_closed_sets}).
\end{lemma}

Note that $T^\pp\setminus B^\pp$ may not be a tree system and so we have not defined $[T^\pp\setminus B^\pp]$.

\begin{proof}
	As discussed above, the fact that~$T^\pp$ is balanced implies that $[T^\pp]$ is closed. That $X^\qq\subseteq X^\pp$ when~$\qq$ extends~$\pp$ again follows directly from the definition of the partial ordering on~$\PP_2$.

	\smallskip

Let $\pp\in \PP_2$. Suppose that $[T^\pp]\subseteq [B^\pp]^\prec$. Since~$T^\pp$ is bounded, $[T^\pp]$ is compact. There is some finite~$C\subset B^\pp$ such that $[T^\pp]\subseteq [C]^\prec$. We may assume that~$C$ is prefix-free. Then~$C$ shows that~$B^\pp$ is $(h^\pp,h^\pp)$- and so $(b^\pp,b^\pp)$-big above $(\s^\pp,\mu^\pp)$. Hence $X^\pp$ is nonempty.

\smallskip

Let~$m<\w$. Since $h^\pp\ge b^\pp$ above $\min \{ |\s^\pp|,|\mu^\pp|\}$ \cref{lem:two_step:can_take_full_subtree} shows that there is some pair $(\tau,\rho)\in T^\pp$ with $|\tau|,|\rho|\ge m$ above which $B^\pp$ is $(b^\pp,b^\pp)$-small. Then
 	$   \qq = ((\tau,\rho),T^\pp \cap (\tau,\rho)^{\preceq},B^\pp\cap (\tau,\rho)^{\preceq},h^\pp,b^\pp)   $
is a condition in~$\PP_2$ extending~$\pp$ satisfying $X^\qq\subseteq [T^\qq] \subseteq [\tau,\rho]^\prec$. 
Thus for every~$m$, the set of conditions~$\pp\in \PP_2$ such that $X^\pp\subseteq [\tau,\rho]^{\prec}$ for some strings~$\tau,\rho$, both of length at least~$m$, is dense in~$\PP_2$; this implies requirement~(c) of \cref{def:forcing_with_closed_sets}.
\end{proof}

As in the previous section, if~$\GG\subset \PP_2$ is sufficiently generic then $\bigcap_{\pp\in \GG} [T^\pp]\setminus [B^\pp]^\prec$ is a singleton which we denote by $\{(x^\GG,y^\GG)\}$. In fact
 $x^\GG = \bigcup \left\{ \s^\pp  \,:\, \pp\in \GG \right\}$ and
 $y^\GG = \bigcup \left\{ \mu^\pp  \,:\, \pp\in \GG \right\}$.

Let~$\pp\in \PP_2$; since $B_{\DNC^2} \subseteq B^\pp$ we see:

\begin{proposition} \label{prop:TwoStep:DNC}
	Every condition in~$\PP_2$ forces that $x^\GG\in \DNC$ and that $y^\GG\in \DNC^{x^\GG}$.
\end{proposition}

\subsubsection*{The restriction of~$\PP_2$ to~$\PP_1$} 
\label{ssub:the_restriction_of_P2_to_P1}

We do not actually have a restriction map to~$\PP_1$ from~$\PP_2$ but from a dense subset of~$\PP_2$. Note that if $\QQ\subseteq \PP$ is dense and $G\subset \QQ$ is a generic directed set, then it is also a generic directed subset of~$\PP$.

\begin{proposition} \label{lem:TwoStep:the_restriction_map}
	There is a dense subset~$\QQ_2\subseteq \PP_2$ and a restriction map $i\colon \QQ_2 \to \PP_1$  such that for all $\pp\in \QQ_2$, $X^{i(\pp)}\supseteq \dom X^\pp$.
\end{proposition}

In particular this shows that~$\PP_2$ is nonempty.

\begin{proof}
We define $i\colon \PP_2\to \PP_1$ by letting
\[ i(\qq)= (\s^\qq, \dom T^\qq, \project{b^\qq}{\mu^\qq}{B^\qq}, h^\qq,b^\qq)\]where we recall that $\project{b^\qq}{\mu^\qq}{B^\qq}$ is the set of~$\tau$ such that $B^\qq(\tau)$ is $b^\qq$-big above~$\mu^\qq$.

Let~$\qq\in \PP_2$. It is routine to check that $i(\qq)\in \PP_1$.

However, we cannot show that~$i$ is order-preserving. For this reason we let
		\[
			\QQ_2  = \left\{ \pp\in \PP_2 \,:\,  \project{b^\pp}{\mu^\pp}{B^\pp} = \left\{ \tau\in \dom T^\pp \,:\, \mu^\pp\in B^\pp(\tau)  \right\} \right\}.
		\]
	Suppose that~$\qq\in \QQ_2$; then $X^{i(\qq)} \supseteq \dom X^\qq$. To check this we observe that if $(x,y)\in [T^\qq]\setminus [B^\qq]^\prec$ then for all $\tau\prec x$, $(\tau,\mu^\qq)\notin B^\qq$ and so $\tau\notin B^{i(\qq)}$; so $x\in [\dom T^\qq]\setminus [B^{i(\qq)}]^\prec$. (In fact $X^{i(\qq)}= \dom X^\qq$; if $x\in X^{i(\qq)}$ then $B^\qq(x)$ is $b^\qq$-small above~$\mu^\qq$, so $T^\qq(x)\setminus B^\qq(x)$ has a path.)

	\smallskip

	Let $\qq\in \PP_2$. Define a set $B\subset T^\qq$: for $\tau\in \dom T^\qq \setminus \project{b^\qq}{\mu^\qq}{B^\qq}$ we let $B(\tau) = B^\qq(\tau)$; for $\tau\in\project{b^\qq}{\mu^\qq}{B^\qq}$
	we let $B(\tau) = T^\qq(\tau)$. Let~$\nu(\qq) = ((\s^\qq,\mu^\qq),T^\qq,B,h^\qq,b^\qq)$. The concatenation property implies that $\project{b^\qq}{\mu^\qq}{B^\qq} = \project{b^\qq}{\mu^\qq}{B}$, which shows that $\nu(\qq)\in \PP_2$, in fact that $\nu(\qq)\in \QQ_2$, and it extends~$\qq$. Hence $\QQ_2$ is dense in~$\PP_2$. We observe that $i(\qq) = i(\nu(\qq))$.

	\smallskip

	To show that the restriction of~$i$ to~$\QQ_2$ is order-preserving we need to check that if~$\qq$ extends~$\ss$ are in~$\QQ_2$ then $B^{i(\ss)}\cap T^{i(\qq)}\subseteq B^{i(\qq)}$. If $\tau\in B^{i(\ss)}$ (and $\tau\in T^{i(\qq)}$) then $(\tau,\mu^\ss)\in B^\ss$; since $B^\ss$ is open in~$T^\ss$, this means that $(\tau, \mu^\qq)\in B^\ss$; since $B^\ss\cap T^\qq\subseteq B^\qq$, $(\tau,\mu^\qq)\in B^\qq$ and so $\tau\in B^{i(\qq)}$.

	\smallskip

	Let~$\qq\in \QQ_2$ and let~$\pp\in \PP_1$ extend $i(\qq)$; we need to find $\rr\in \QQ_2$ extending~$\qq$ such that $i(\rr)$ extends~$\pp$. Using the map~$\nu$, it suffices to find $\rr\in \PP_2$.

	Let $T$ be the restriction of~$T^\qq$ to~$T^\pp$: $\dom T = T^\pp$ and for $\tau\in T^\pp$, $T(\tau) = T^\qq(\tau)$. The system~$T$ is $(h^\pp,h^\qq)$-bushy above $(\s^\pp,\mu^\qq)$.

	Also define $B\subseteq T$; if $\tau\in B^\pp$ then $B(\tau) = T(\tau)$; if $\tau\in T^\pp\setminus B^\pp$ then $B(\tau) = B^\qq(\tau)$. The set~$B$ is open in~$T$, is c.e., and is $(b^\pp,b^\qq)$-small above $(\s^\pp,\mu^\qq)$. To see that let~$S$ be $(b^\pp,b^\qq)$ bushy above $(\s^\pp,\mu^\qq)$; by \cref{rmk:two_step:minimal_bigness_witness} we may assume that $S\subset T$. Since $\dom S$ is a subtree of~$T^\pp$ we find a leaf~$\tau$ of $\dom S$ which is not in~$B^\pp$. Since~$\pp$ extends~$i(\qq)$, $\tau\notin B^{i(\qq)}$ and so $B(\tau) = B^\qq(\tau)$ is $b^\qq$-small above~$\mu^\qq$, so $S(\tau)$ has a leaf $\rho$ which is not in $B(\tau)$.

	Since $h^\pp\ge b^\pp$ above~$|\s^\pp|$ and $h^\qq\ge b^\qq$ above $|\mu^\qq|$, $T$ is $(b^\pp,b^\qq)$-bushy. By \cref{lem:two_step:can_take_full_subtree} we can find $(\s,\mu)\in T$ such that $|\s|,|\mu| \ge \max\{ |\s^\pp|,|\mu^\qq|\}$ and above which~$B$ is $(b^\pp,b^\qq)$-small.

	We now define $\rr = ((\s,\mu), T\cap (\s,\mu)^\preceq, B\cap (\s,\mu)^\preceq, h^\pp,b^\pp)$. The point is that $h^\pp\le h^\qq$ and $b^\pp\ge b^\qq$ above $\min \{|\s|,|\mu|\}$ and so~$T^\rr$ is $(h^\pp,h^\pp)$-bushy and $B^\rr$ is $(b^\pp,b^\pp)$-small above $(\s,\mu)$. This also shows that $\rr$ extends~$\qq$. To show that $i(\rr)$ extends~$\pp$ we need to show that $B^\pp\cap \dom T^\rr\subseteq B^{i(\rr)}$. Let $\tau\in B^\pp\cap \dom T^\rr$. Then $\tau\succeq \s$ and so $\mu\in T(\tau) = B(\tau)$, so $\tau\in B^{i(\rr)}$.
\end{proof}

\begin{corollary} \label{cor:two_step:the_first_coordinate_is_as_desired}
	Every condition in~$\PP_2$ forces that~$x^\GG$ has minimal Turing degree.
\end{corollary}


\subsubsection*{Totality} 
\label{subsub:two_step:totality}

\begin{proposition} \label{prop:two_step:forcing_Pi_2}
	Let~$C\subseteq \Baire{2}$ be~$\Pi^0_2$ and let~$\pp\in \PP_2$. If $\pp\force (x^\GG,y^\GG)\in C$ then~$\pp$ has an extension which strongly forces that~$(x^\GG,y^\GG)\in C$.
\end{proposition}

\begin{proof}
	The proof is similar to the proof of \cref{prop:one_step:forcing_Pi_2}. We choose a function~$g\in \Quick$ such that $h^\pp\gg g\gg b^\pp$. By \cref{lem:two_step:can_take_full_subtree} we may assume that $h^\pp \ge g\ge b^\pp$ above $\min \{|\s^\pp|,|\mu^\pp|\}$.

	We fix a sequence of c.e.\ sets~$C_k\subseteq T^\pp$, open in~$T^\pp$, such that $C\cap [T^\pp] = [T^\pp]\cap \bigcap_k [C_k]^\prec$. For all $(\tau,\rho)\in T^\pp$, for all~$k$, the set $B^\pp\cup C_k$ is $(g,g)$-big above~$(\tau,\rho)$; otherwise $((\tau,\rho),T^\pp\cap (\tau,\rho)^\preceq, (B^\pp\cup C_k)\cap (\tau,\rho)^\preceq, h^\pp,g)$ is a condition extending~$\pp$ which forces that $(x^\GG,y^\GG)\notin C$.

	We define a sequence of finite tree systems $S_k\subset T^\pp$ such that: each~$S_k$ is $(g,g)$-bushy; $S_{k+1}$ is a proper end-extension of~$S_k$; the leaves of~$S_{k+1}$ are in $C_k\cup B^\pp$; if $k>0$ then there is some $\ell_k$ such that for every $k\ge 1$, for every leaf~$(\tau,\rho)$ of~$S_k$, $|\tau|= |\rho| = \ell_k$. We begin with $S_0 = \{(\s^\pp,\mu^\pp)\}$. Given~$S_k$, \cref{lem:two_step:restricted_weak_concatentation} says that $C_k\cup B^\pp$ is $(g,g)$-big above the set of leaves of~$S_k$, so we can find a finite $(g,g)$-bushy end-extension $S'_k\subset T^\pp$ of~$S_k$ with leaves in $C_k\cup B^\pp$.

	Now find some $\ell_{k+1}$, greater than $|\tau|$ and $|\rho|$ for any leaf $(\tau,\rho)$ of~$S'_k$, which is a balanced level for~$T^\pp$ (\cref{def:two_step:balanced_tree_systems}). Then the set of $(\tau,\rho)\in T^\pp$ such that $|\tau|=|\rho|=\ell_{k+1}$ is $(g,g)$-big above the set of leaves of~$S'_k$. Hence we can find~$S_{k+1}\subset T^\pp$ to be an end-extension of~$S'_k$ as required.

	It follows that $S = \bigcup_k S_k$ is a computable, $(g,g)$-bushy and balanced tree system above $(\s^\pp,\mu^\pp)$ and that the condition $((\s^\pp,\mu^\pp),S,B^\pp\cap S, g,b^\pp)$ extends~$\pp$ and strongly forces that $(x^\GG,y^\GG)\in C$.
\end{proof}


\subsection{Minimal cover} 
\label{sub:two_step:minimal_cover}

We work toward showing that $\deg_\Tur(x^\GG,y^\GG)$ is a strong minimal cover of $\deg_\Tur(x^\GG)$. We do this in two steps. First we show that it is a minimal cover. This mostly uses the tools of the previous section.

Let~$\Gamma\colon \Baire{2}\to 2^\w$ be a Turing functional. For a condition $\pp\in \PP_2$, a bounding function~$g$ and a string~$\mu$ let $\Split{\Gamma}{g}{\mu}{\pp}$ be the set of $\tau\in \dom T^\pp$ such that $T^\pp(\tau)$ contains two sets~$A_0(\tau)$ and~$A_1(\tau)$, both $g$-big above~$\mu$, which $\Gamma(\tau,-)$-split mod $B^\pp(\tau)$.

\begin{lemma} \label{lem:two_steps:finding_local_splittings}
	Suppose that $\pp\in \PP_2$ strongly forces that $\Gamma(x^\GG,y^\GG)$ is total and forces that~$\Gamma(x^\GG,y^\GG)\nle_\Tur x^\GG$.
	
	Let~$(\s,\mu)\in T^\pp$. Let $g\in\Quick$  such that
	 $h^\pp\gg g$, and $h^\pp\ge 3g$ and $ g\ge b^\pp$ above $\min \{|\s|,|\mu|\}$. Then $\Split{\Gamma}{g}{\mu}{\pp}$ is $g$-big above~$\s$.
\end{lemma}

\begin{proof}
	Suppose that~$(\s,\mu)$ and~$g$ witness the failure of the lemma; we find an extension of~$\pp$ which forces that~$\Gamma(x^\GG,y^\GG)$ is computable from~$x^\GG$.

	Let~$\Theta$ be the (c.e.) set of pairs $(\tau,\alpha)$ such that $\tau\in \dom T^\pp$, $\alpha\in 2^{<\w}$ and $A_\alpha(\tau)$ is $g$-big above~$\mu$, where as before $A_\alpha = B^\pp\cup \{(\tau,\rho)\in T^\pp\,:\, \Gamma(\tau,\rho)\succeq \alpha\}$.

	For brevity let $C = \Split{\Gamma}{g}{\mu}{\pp}$. The set~$C$ is open in~$\dom T^\pp$. If $\tau\in \dom T^\pp\setminus C$ then the strings in $\Theta(\tau)$ are pairwise comparable.

	Let $\tau\in \dom T^\pp\setminus C$. The argument of the proof of \cref{lem:one_step:finding_splittings} shows that if $|\Gamma(\tau,\rho)|\ge m$ for every leaf~$\rho$ of $T^\pp(\tau)$ which is not in~$B^\pp(\tau)$ then $\Theta(\tau)$ contains a string of length~$m$. Also, $B^\pp(\tau)$ is $g$-small above~$\mu$ and so $\Theta(\tau)$ is finite; in this case we let $\Theta^\tau= \bigcup \Theta(\tau)$ be the longest string in $\Theta(\tau)$.

	If $\tau\preceq \tau'$ are in $\dom T^\pp\setminus C$ then $\Theta^\tau\preceq \Theta^{\tau'}$. This follows from the fact that $A_\alpha(\tau)\subseteq A_\alpha(\tau')$ for all~$\alpha$.

	\smallskip

	Let $ D \,= \{(\tau,\rho)\in T^\pp\,:\, \tau\in C \text{ or }  \Gamma(\tau,\rho)\perp \alpha \text{ for some }\alpha\in \Theta(\tau)\}.$
The set~$D$ is c.e.\ and is open in~$T^\pp$. Also, $D\cup B^\pp$ is $(g,g)$-small above $(\s,\mu)$. To see this, suppose that $S\subset T^\pp$ is a finite $(g,g)$-bushy tree system above $(\s,\mu)$ (as above we use \cref{rmk:two_step:minimal_bigness_witness}). Then there is a leaf~$\tau$ of $\dom S$ which is not in~$C$; and then $S(\tau)$ must contain a leaf~$\rho\notin B^\pp(\tau)$ such that $\Gamma(\tau,\rho)$ is compatible with $\Theta^\tau$.

Now suppose that~$(x,y)\in [T^\pp]\setminus [D\cup B^\pp]^\preceq$. No initial segment of~$x$ is in~$C$. A compactness argument shows that $\Theta(x) = \bigcup_{\tau\prec x} \Theta^\tau$ is total, and so $\Gamma(x,y)= \Theta(x)$. Certainly $\Theta(x)\le_\Tur x$. Therefore the condition $((\s,\mu), T^\pp\cap (\s,\mu)^\preceq, (D\cup B^\pp)\cap (\s,\mu)^\preceq, h^\pp,g)$ extends~$\pp$ and (strongly) forces that $\Gamma(x^\GG,y^\GG)\le_\Tur x^\GG$.
\end{proof}

\begin{definition} \label{def:two_step:local_splittings}
	Let~$B\subseteq \Strings{2}$. Two sets~$A_0$ and~$A_1$ of pairs of strings \emph{locally $\Gamma$-split mod~$B$} if for all~$\tau$, $A_0(\tau)$ and $A_1(\tau)$ form a $\Gamma(\tau,-)$-splitting mod $B(\tau)$. That is, if $(\tau,\rho_0)\in A_0\setminus B$ and $(\tau,\rho_1)\in A_1\setminus B$ then $\Gamma(\tau,\rho_0)\perp \Gamma(\tau,\rho_1)$.
\end{definition}

We introduce the notion of uniform largeness.

\begin{definition} \label{def:two_step:uniform_bigness}
	Let~$A$ be finite and prefix-free, and let $\mathcal B$ be a collection of sets of pairs of strings. We say that the sets in~$\mathcal B$ are \emph{uniformly $(g,h)$-big above~$A$} if the set of~$\tau$ such that for all ${B\in \mathcal B}$, $B(\tau)$ is $h$-big above~$A(\tau^{-\dom A})$, is~$g$-big above~$\dom A$.
\end{definition}

The conclusion of \Cref{lem:two_steps:finding_local_splittings} is that there are~$A_0$ and~$A_1$, subsets of~$T^\pp$ uniformly $(g,g)$-big above~$(\s,\mu)$, which locally $\Gamma$-split mod~$B^\pp$.

\begin{lemma} \label{lem:two_steps:finding_many_local_splittings}
	Suppose that $\pp\in \PP_2$ strongly forces that $\Gamma(x^\GG,y^\GG)$ is total and forces that~$\Gamma(x^\GG,y^\GG)\nle_\Tur x^\GG$.
	
	Let~$\s\in\dom T^\pp$, and let $\mu_1,\mu_2,\dots, \mu_k$ be elements of $T^\pp(\s)$. Let $g\in\Quick$  such that $h^\pp\gg g$, and $h^\pp\ge 3^kg$ and $g\ge b^\pp$ above~$\min \{|\s|,|\mu_1|, |\mu_2|, \dots, |\mu_k|\}$. Then there is a set~$A\subset T^\pp$, $(g,g)$-big above $\{(\s,\mu_j)\,:\, j\le k\}$, such that the sets $A\cap (\s,\mu_j)^\preceq$ pairwise locally $\Gamma$-split mod~$B^\pp$.
\end{lemma}

\begin{proof}
	The idea is to extend bushily on the first coordinate so that we can emulate the proof of \cref{lem:one_step:finding_many_splittings} on the second coordinate. Formally this is done by induction on~$k$. Suppose this has been shown for~$k$; let $\mu_1,\dots, \mu_k$ and $\mu^*$ be elements of $T^\pp(\s)$; suppose that $h\gg g$, and $h^\pp\ge 3^{k+1}g$ and $g\ge b^\pp$ above~$\min  \{|\s|,|\mu^*|,|\mu_j|\,:\, j\le k\}$. Then $h\gg 3g$; so by induction we can find a set~$A$, $(3g,3g)$-big above $\{(\s,\mu_j)\,:\, j\le k\}$ such that the sets $A\cap (\s,\mu_j)^\preceq$ pairwise locally~$\Gamma$-split mod~$B^\pp$. In fact we only need $(g,3g)$-big.

	 Let $(\zeta,\nu)\in A $.  By \cref{lem:two_steps:finding_local_splittings}, for all $\zeta'\succeq \zeta$ on~$\dom T^\pp$, $\Split{\Gamma}{3g}{\nu}{\pp}$ is $3g$-big above $\zeta'$ (again we only need~$g$-big). By repeatedly extending we see that for all~$\zeta\in \dom A$, $Q_\zeta = \zeta^\preceq \cap \bigcap_{\nu\in A(\zeta)} \Split{\Gamma}{3g}{\nu}{\pp}$ is~$3g$-big above~$\zeta$. We extend the set~$A$ by letting $A(\tau) = A(\zeta)$ for all $\tau\in Q_\zeta$. Let~$Q = \bigcup_{\zeta\in \dom A} Q_\zeta$; it is $3g$-big above~$\s$. For every~$\tau\in Q$ and all $\nu\in A(\tau)$ we can find sets $E_{\nu,0}(\tau), E_{\nu,1}(\tau)\subset T^\pp(\tau)$, each $3g$-big above~$\nu$, which $\Gamma(\tau,-)$-split mod $B^\pp(\tau)$.

	 Further, by extending in $\dom T^\pp$, we may assume that for all $\tau\in Q$ we can find $F(\tau)\subset T^\pp(\tau)$ which is $3^{k}g$-big above~$\mu^*$ and such that $|\Gamma(\tau,\rho)|> |\Gamma(\tau, \eta)|$ for all $\rho\in F(\tau)\setminus B^\pp(\tau)$ and all $\eta\in E_{\nu,i}(\tau)$ (for both $i<2$ and all $\nu\in A(\tau)$).

	 Overall we see that for all~$\tau\in Q$ we can run the argument proving \cref{lem:one_step:finding_many_splittings} inside $T^\pp(\tau)$ and using \cref{lemma:one_step:combinatorial} find $F'(\tau)\subseteq F(\tau)$, $g$-big above~$\mu^*$ and for $j\le k$, $E'_j(\tau)\subset T^\pp(\tau)$, $g$-big above~$\mu_j$, with every string in~$E'_j(\tau)$ extending some string in $A_j(\tau)$, such that $F'(\tau)$ and $E'_j(\tau)$ form a $\Gamma(\tau,-)$-splitting mod $B^\pp(\tau)$; the fact that strings in~$E'_j(\tau)$ extend strings in $A(\tau)\cap \mu_j^\preceq$ shows that the sets $E'_j(\tau)$ also pairwise $\Gamma(\tau,-)$-split mod $B^\pp(\tau)$.
\end{proof}

\begin{proposition} \label{lem:two_step:forcing_a_minimal_cover}
Every condition in~$\PP_2$ forces that if $\Gamma(x^\GG,y^\GG)$ is total and $\Gamma(x^\GG,y^\GG)\nle_\Tur x^\GG$ then $\Gamma(x^\GG,y^\GG)\oplus x^\GG\ge_\Tur y^\GG$.
\end{proposition}

\begin{proof}
	As in the proof of \cref{prop:one_step:splitting_tree} we take some $\pp\in \PP_2$ which strongly forces that $\Gamma(x^\GG,y^\GG)$ is total and forces that $\Gamma(x^\GG,y^\GG)\nle_\Tur x^\GG$, and find an extension of~$\pp$ which forces that $\Gamma(x^\GG,y^\GG)\oplus x^\GG\ge_\Tur y^\GG$.

	Find some $g\in \Quick$ such that $h^\pp\gg g\gg b^\pp$. Let $\bar g(m) = \prod_{k<m} g(k)$. By \cref{lem:two_step:can_take_full_subtree} we can extend $(\s^\pp,\mu^\pp)$ so that $h^\pp \ge 3^{\bar g}g$ and $g\ge b^\pp$ above $\min \{|\s^\pp|,|\mu^\pp|\}$.

	We define an increasing sequence~$\seq{\ell_k}$ and a sequence $\seq{S_k}$ of finite subsystems of~$T^\pp$ such that: $\dom S_k$ is $g$-bushy and for all $\tau\in \dom S_k$, $S_k(\tau)$ is exactly $g$-bushy; $S_{k+1}$ is a proper end-extension of~$S_k$; for every leaf~$(\tau,\rho)$ of~$S_k$, $|\tau| = |\rho| = \ell_k$.

	To begin we find some $\ell_0> |\s^\pp|, |\mu^\pp|$, a balanced level for~$T^\pp$. We let $\dom S_0 = \dom T^\pp\rest {\w^{\le \ell_0}}$ and for each leaf~$\tau$ of~$\dom S_0$ we let $S_0(\tau)$ be an exactly $g$-bushy subtree of~$T^\pp(\tau)$ whose leaves all have lenght~$\ell_0$. As usual if $\tau\in \dom S_0$ is not a leaf then we let $S_0(\tau) = \{\mu^\pp\}$.

	Given~$S_k$ we note that for every leaf~$\s$ of $\dom S_k$, the number of leaves of~$S_k(\s)$ is precisely $\prod_{m\in [|\mu^\pp|,\ell_k)} g(m)$ which is bounded by $\bar g(\ell_k)$; and $h^\pp\ge 3^{\bar g(\ell_k)} g$ above~$\ell_k$. By \cref{lem:two_steps:finding_many_local_splittings} we can find for each leaf~$\s$ of~$\dom S_k$ a finite $(g,g)$-bushy forest system $R_\s\subset T^\pp$ above~$\left\{ (\s,\nu) \,:\,  \nu \text{ a leaf of }S_k(\s) \right\}$, such that for every leaf~$\tau$ of~$\dom R_\s$, the sets $R_\s(\tau)\cap \nu^\preceq$ for the leaves~$\nu$ of $S_k(\s)$ pairwise $\Gamma(\tau,-)$-split mod~$B^\pp$. By shrinking we may assume that for all leaves~$\tau\in \dom R_\s$, $R_\s(\tau)$ is exactly $g$-bushy. Let~$R = \bigcup_\s R_\s$ and let $S'_k = S_k\conc R$.

	Now as in the proof of \cref{prop:two_step:forcing_Pi_2} we let~$\ell_{k+1}$ be a balanced level of~$T^\pp$, greater than the length of any string appearing in ~$S_{k}'$, and let~$S_{k+1}\subset T^\pp$ be an end-extension of~$S'_k$ with the desired properties.

	Let $S = \bigcup_k S_k$. Then for all $x\in [\dom S]$, $\Gamma(x,-)$ is 1-1 on $[S(x)]\setminus [B^\pp(x)]^\prec$. The tuple $((\s^\pp,\mu^\pp),S, B^\pp\cap S, g,b^\pp)$ is a condition as required (relativise \cref{lem:1-1_or_constant_on_compact_space} to each~$x$).
\end{proof}


\subsection{Strong minimal cover} 
\label{sub:strong_minimal_cover}

The following is the usual definition of splitting, restated for pairs of strings.

\begin{definition} \label{def:two_step:global_splittings}
	Let~$B\subseteq \Strings{2}$. Two sets~$A_0$ and~$A_1$ \emph{$\Gamma$-split mod~$B$} if for all $(\tau,\rho)\in A_0\setminus B$ and $(\tau',\rho')\in A_1\setminus B$, $\Gamma(\tau,\rho)\perp \Gamma(\tau',\rho')$.
\end{definition}

\begin{lemma} \label{lemma:two_steps:combinatorial}
	Let~$g_1,g_2,h_1,h_2\in \Quick$; let~$B$ be an open set of pairs of strings. Suppose that:
	\begin{itemize}
		\item $(\s,\mu)$ and $(\s^*,\mu^*)$ are pairs of strings;
		\item $A$ is $(3g_1,3g_2)$-big above~$(\s,\mu)$;
		\item $E_0$ and~$E_1$ are uniformly $(3g_1,3g_2)$-big above~$A$; and for all $(\tau,\rho)\in A$, $E_0\cap (\tau,\rho)^\preceq$ and $E_1\cap (\tau,\rho)^\preceq$ locally~$\Gamma$-split mod~$B$; and
		\item $F$ is $(3h_1,3h_2)$-big above $(\s^*,\mu^*)$, and $|\Gamma(\lambda,\nu)|> |\Gamma(\zeta,\eta)|$ for all $(\lambda,\nu)\in F\setminus B$ and all $(\zeta,\eta)\in E\setminus B$, where $E = E_0\cup E_1$.
	\end{itemize}
	Then there are $E'\subseteq E$, $(g_1,g_2)$-big above~$(\s,\mu)$, and $F'\subseteq F$, $(h_1,h_2)$-big above~$(\s^*,\mu^*)$, which $\Gamma$-split mod~$B$.
\end{lemma}		

\begin{proof}
	The proof is very similar to that of \cref{lemma:one_step:combinatorial}. As above, for a string $\alpha\in 2^{<\w}$ let
	 $F_{\succeq \alpha} = (F\cap B) \cup \left\{ (\tau,\rho)\in F  \,:\, \Gamma(\tau,\rho)\succeq \alpha \right\} ,$
	and similarly define $F_{\perp\alpha}$, $E_{\succeq\alpha}$, $E_{\preceq\alpha}$ and so on.
		If $F\cap B$ is $(h_1,h_2)$-big above~$(\s^*,\mu^*)$ then we can let~$F'=F\cap B$ and $E' = E$. Similarly if $E\cap B$ is~$(g_1,g_2)$-big above~$(\s,\mu)$.
	
	 Suppose otherwise. In that case, for sufficiently long~$\alpha$, $F_{\succeq \alpha}$ is $(h_1,h_2)$-small above~$(\s^*,\mu^*)$. Let~$\alpha$ be a string, maximal with respect to $F_{\succeq \alpha}$ being~$(h_1,h_2)$-big above~$(\s^*,\mu^*)$. As above we show that either
	\begin{enumerate}
		\item $E_{\perp\alpha}$ is $(g_1,g_2)$-big above~$(\s,\mu)$, or
		\item $E_{\succeq \alpha}$ is $(g_1,g_2)$-big above~$(\s,\mu)$ and $F_{\perp\alpha}$ is $(h_1,h_2)$-big above~$(\s^*,\mu^*)$.
	\end{enumerate}
	In both cases we can find~$E'$ and~$F'$ as required.
	
	Again we examine two cases, depending on $E_{\preceq \alpha}$.
	
	\smallskip
	
	First suppose that~$E_{\preceq \alpha}$ is $(g_1,g_2)$-big above~$(\s,\mu)$. Let~$R$ witness this. Fix $\zeta$, a leaf of $\dom R$. The argument of the proof of \cref{lemma:one_step:combinatorial} is now carried out within $R(\zeta)$. Let $\tau = \zeta^{-\dom A}$. Every $\nu\in E(\zeta)$ extends some unique $\rho\in A(\tau)$. The tree $R(\zeta)$ restricted to initial segments of strings in~$A(\tau)$ shows that $A(\tau)\cap R(\zeta)$ is $g_2$-big above~$\mu$; for each $\rho\in A(\tau)\cap R(\zeta)$, $E_{\preceq \alpha}(\zeta)$ is $g_2$-big above~$\rho$. The previous argument shows that for each such~$\rho$,  $E_{\perp \alpha}(\zeta)$ is $g_2$-big above~$\rho$. The concatenation property shows that $E_{\perp \alpha}(\zeta)$ is $g_2$-big above~$\mu$. And then $\dom R$ shows that $E_{\perp\alpha}$ is $(g_1,g_2)$-big above~$(\s,\mu)$.
	
\smallskip
	
	Next suppose that~$E_{\preceq\alpha}$ is $(g_1,g_2)$-small above~$(\s,\mu)$; the argument is now identical to the comparable one in \cref{lemma:one_step:combinatorial}, using \cref{lem:TwoSteps:BigSubset}. It shows that~(2) holds. 	
\end{proof}

\begin{lemma} \label{lem:two_steps:getting_many_global_splittings}
	Suppose that $\pp\in \PP_2$ strongly forces that $\Gamma(x^\GG,y^\GG)$ is total and forces that~$\Gamma(x^\GG,y^\GG)\nle_\Tur x^\GG$.

	Let $C\subset T^\pp$ be prefix-free and finite; let $g\in\Quick$ such that
	 $h^\pp\gg g$, and $h^\pp\ge 3^{|C|^2}g$ and $ g\ge b^\pp$ above $\min \{|\s|,|\mu|\,:\, (\s,\mu)\in C\}$.

Then there is a set $A\subset T^\pp$, $(g,g)$-big above~$C$, such that the sets $A\cap (\s,\mu^\pp)^\preceq$ (for $\s\in \dom C$) pairwise $\Gamma$-split mod~$B^\pp$.
\end{lemma}

\begin{proof}
We prove the lemma by induction on $|C|$. Let~$C^* = C\cup \{(\s^*\mu^*)\}\subset T^\pp$ be finite and prefix-free, and suppose that the lemma is already known for~$C$. Let~$g$ satisfy the assumptions of the lemma for~$C^*$. The assumptions of the lemma hold for the set~$C$ and the function $3^{|C|}g$. Let~$A$ be as guaranteed by the lemma for~$C$ and $3^{|C|}g$.

Let $(\s_1,\mu_1), (\s_2,\mu_2),\dots, (\s_k,\mu_k)$ list the elements of~$C$ such that $\s_j\ne \s^*$. By reverse recursion on $j\le k$ we define a set $A_{j}\subset T^\pp$, $(3^jg,3^jg)$-big above $C^*$. We will ensure that $A_j\cap C^\preceq \subset A^\preceq$, and so the sets $A_j\cap (\s,\mu^\pp)^\preceq$ for $\s\in \dom C$ pairwise $\Gamma$-split mod~$B^\pp$. Further, we will ensure that $A_{j-1}\cap (\s^*,\mu^*)^\preceq$ and $A_{j-1}\cap (\s_j,\mu_j)^\preceq$ $\Gamma$-split mod~$B^\pp$; and that $A_{j-1}\subset A_{j}^\preceq$. Thus in the end, the set $A_{0}$ is as required.

We start with $A_{k} = A\cup \{ (\s^*,\mu^*)\}$. Now suppose that $j>0$ and we are given the sets $A_{j}$. Let $\tau\in (\dom A_j)\cap \s_j^\preceq$. \Cref{lem:two_steps:finding_local_splittings} says that for all~$\tau'\succeq\tau$ in $\dom T^\pp$, for all $\rho\in A_j(\tau)\cap \mu_j^\preceq$, the set $\Split{\Gamma}{3^jg}{\rho}{\pp}$ is $3^jg$-big above~$\s_j$. So applying \cref{lem:two_step:one_step_induction_for_weak_concatenation} to these sets, and repeating this process for all such~$\tau$, we find (finite) $E_{j,0}\subset T^\pp$ and $E_{j,1}\subset T^\pp$, uniformly $(3^jg,3^jg)$-big above $A_{j}\cap (\s_j,\mu_j)^\preceq$, such that for every $(\tau,\rho)\in A_{j}\cap (\s_j,\mu_j)^\preceq$, $E_{j,0}\cap (\tau,\rho)^\preceq$ and $E_{j,1}\cap (\tau,\rho)^\preceq$ locally $\Gamma$-split mod~$B^\pp$. Given $E_j = E_{j,0}\cup E_{j,1}$ we can find $F_j\subset T^\pp$, $(3^jg,3^jg)$-big above~$A_{j}\cap (\s^*,\mu^*)^\preceq$ (and lying above that set) such that $|\Gamma(\tau,\rho)|>|\Gamma(\tau',\rho')|$ for all $(\tau,\rho)\in  F_j\setminus B^\pp$ and all $(\tau',\rho')\in E_j$. We then appeal to \cref{lemma:two_steps:combinatorial} with~$F_j$ in the role of~$F$, $E_{j,i}$ in the role of~$E_i$, $A_{j}\cap (\s_j,\mu_j)^\preceq$ in the role of~$A$, and using the function~$3^{j-1}g$ we get $F'_j\subseteq F_j$, $(3^{j-1}g,3^{j-1}g)$-big above~$(\s^*,\mu^*)$ and $E'_j\subseteq E_j$, also $(3^{j-1}g,3^{j-1}g)$-big above~$(\s_j,\mu_j)$, which $\Gamma$-split mod~$B^\pp$.

We now define the set $A_{j-1}$. We first define~$\dom A_{j-1}$, and we do this by defining $(\dom A_{j-1})\cap \s^\preceq$ for all $\s\in \dom C^*$. Let $\s\in \dom C^*$. If $\s\ne \s_j,\s^*$ then $(\dom A_{j-1})\cap \s^\preceq = (\dom A_{j})\cap \s^\preceq$. We let $(\dom A_{j-1})\cap (\s^*)^\preceq = \dom F'_j$ and $(\dom A_{j-1})\cap (\s_j)^\preceq = \dom E'_j$.

Now for $\tau\in \dom A_{j-1}$ we define $A_{j-1}(\tau)$. Fix such~$\tau$; let  $\zeta = \tau^{-\dom A_j}$ and let $\s = \tau^{-\dom C^*} = \zeta^{-\dom C^*}$. If $\s\ne \s^*,\s_j$ then $\zeta = \tau$ and we let $A_{j-1}(\tau) = A_j(\tau)$. Otherwise, we define $A_{j-1}(\tau)$ by defining $A_{j-1}\cap \mu^\preceq$ for all $\mu\in C^*(\s)$. Suppose that~$\s = \s^*$. If $\mu\ne \mu^*$ then we let $A_{j-1}(\tau)\cap \mu^\preceq= A_j(\zeta)\cap \mu^\preceq$ (which inductively will just equal $A(\tau^{-\dom A})\cap \mu^\preceq$). We let $A_{j-1}(\tau)\cap (\mu^*)^\preceq = F'_j(\tau)$. Similarly, if $\s= \s_j$ and $\mu\ne \mu_j$ then we let $A_{j-1}(\tau)\cap \mu^\preceq= A_j(\zeta)\cap \mu^\preceq$; we let $A_{j-1}(\tau)\cap (\mu_j)^\preceq = E'_j(\tau)$.
\end{proof}

\begin{proposition} \label{prop:two_step:forcing_a_strong_minimal_cover}
	Every condition in~$\PP_2$ forces that if $\Gamma(x^\GG,y^\GG)$ is total and $\Gamma(x^\GG,y^\GG)\nle_\Tur x^\GG$ then $\Gamma(x^\GG,y^\GG)\ge_\Tur x^\GG$.
\end{proposition}

\begin{proof}
The construction is similar to the one in \cref{lem:two_step:forcing_a_minimal_cover,prop:one_step:splitting_tree}. It is here that we really use the fact that~$T^\pp$ is balanced, for we ensure that each~$S_k$ we build is exactly $(g,g)$-bushy. We assume that $h^\pp \gg 3^{{\bar g}^2}g$ above $\min \{|\s^\pp|,|\mu^\pp|\}$ and then apply \cref{lem:two_steps:getting_many_global_splittings} to~$C$ being the set of leaves of~$S_k$. We use \cref{lem:1-1_on_sections_of_compact}.
\end{proof}

And as a result:

\begin{proposition} \label{prop:two_step:minimality}
	Every condition in~$\PP_2$ forces that $\deg_\Tur(x^\GG,y^\GG)$ is a strong minimal cover of~$\deg_\Tur(x^\GG)$.
\end{proposition}


\begin{remark} \label{rmk:two_step:totally_splitting}
	We could combine the proofs of \cref{lem:two_steps:finding_many_local_splittings,lem:two_steps:getting_many_global_splittings} to build a ``totally $\Gamma$-splitting'' extension: a set~$A$ such that if $(\s_i,\mu_i)\in C$ (for $i<2$) and $(\tau_i,\rho_i)\in A\cap (\s_i,\mu_i)^\preceq \setminus B$, then $\Gamma(\tau_0,\rho_0)\perp \Gamma(\tau_1,\rho_1)$ provided that either $\s_0\ne \s_1$, or $\tau_0=\tau_1$ (and $\rho_0\ne \rho_1$). We could then have a single construction (replacing \cref{prop:two_step:minimality,prop:two_step:forcing_a_strong_minimal_cover}) giving a condition forcing that $\Gamma(x^\GG,y^\GG)\equiv_\Tur (x^\GG,y^\GG)$.
\end{remark}

\section{The general step} \label{sec:step_n}

We now generalise to get a linearly ordered initial segment of length~$n$. Once the correct definitions are in place, much of the development closely follows the previous section.

\subsection{\boldmath Length~$n$ forest systems} \label{subsec:step_n:forest_systems}

We work with~$n$-tuples of strings. We use boldface notation for tuples. If $\+{\tau}$ is a tuple then $\tau_i$ denotes the $i\tth$ component of~$\+\tau$. The partial ordering of extension~$\preceq$ on~$\Strings{n}$ is defined as expected. For a set~$A\subseteq \Strings{n}$ we let~$A^\preceq$ be the upward closure of~$A$ under this partial ordering. If $\+\tau$ is an $n$-tuple and $k\le n$ then we let $\+\tau\rest{k} = (\tau_1,\dots, \tau_k)$ and $\+\tau\rest{(k,n]} = (\tau_{k+1},\dots, \tau_n)$.
\smallskip

For a set~$A\subseteq \Strings{n}$ and $k<n$ we let $\dom_k A$ be the domain of~$A$ thought of as a relation between $k$-tuples and $(n-k)$-tuples:
\[
	\dom_k A = \left\{  \+{\tau}\rest{k} \,:\,  \+\tau\in  A   \right\}.
\]
For $\+\tau\in \Strings{k}$ we let
\[
	A(\+\tau) = \left\{ \+\rho\in \Strings{n-k} \,:\,  (\+\tau,\+\rho)\in A \right\}.
\]
We will frequently need to chop off the last bit, so for compact notation we let $\chop{\+\tau} = \+\tau\rest{n-1}$ for all $\+\tau\in \Strings{n}$, and let $\chop{A} = \dom_{n-1}A = \left\{ \chop{\+\tau} \,:\,  \+\tau\in A \right\}$ for all $A\subseteq \Strings{n}$.

\begin{definition} \label{def:n_step:prefix-free}
	By induction on~$n$ we define the notion of a \emph{prefix-free} set of tuples of strings: a set $A\subset\Strings{n}$ is prefix-free if $\chop{A}$ is prefix-free, and for all $\+\tau\in \chop{A}$, $A(\+\tau)$ is a prefix-free set of strings.
\end{definition}

If~$A$ is prefix-free and $\+\tau\in A^\preceq$ then there is a unique $\+\s\in A$ such that $\+\s\preceq \+\tau$ (formally this is proved by induction on~$n$); we denote this~$\+\s$ by $\+\tau^{-A}$. Note that if~$A$ is prefix-free and $\+\tau\in A^\preceq$ then $\chop{\+\tau}\in (\chop{A})^\preceq$ and $(\chop{\+\tau})^{-\cchop{A}} = \chop{\+\tau^{-A}}$.

\begin{definition} \label{def:length_n_forest_system}
By induction on~$n$ we define the notion of a length~$n$ forest system. Let~$A\subset \Strings{n}$ be prefix-free and finite. A \emph{length~$n$ forest system above~$A$} is a set~$T\subseteq A^\preceq$ such that:
\begin{itemize}
	\item $\chop{T}$ is a length $n-1$ forest system above~$\chop{A}$;
	\item for all $\+\tau\in \chop{T}$, $T(\+\tau)$ is a finite forest above $A(\+\tau^{-\cchop{A}})$;
	\item if $\+\tau\preceq \+{\tau}'\in \chop{T}$ then $T({\+\tau}')$ is an end-extension of~$T(\+\tau)$.
\end{itemize}
\end{definition}

A forest system~$S$ is a subsystem of~$T$ if $S\subseteq T$. We write $\length{T}$ for the length of~$T$. If~$A$ is a singleton~$\+\s$ then we say that~$T$ is a \emph{tree system} above~$\+\s$.

\begin{lemma} \label{lem:n_step:full_subsystem}
	Let~$T$ be a tree system and let~$\+\s\in T$. Then $T\cap \+\s^\preceq$ is a tree system above~$\s$.
\end{lemma}

(In fact~$\s$ can be replaced by any finite, prefix-free subset of~$T$).

\begin{proof}
	By induction on $\length{T}$. Let~$R = T\cap \s^\preceq$. The point is that $\chop{R} = \chop{T}\cap (\chop{\+\s})^\preceq$. For suppose that $\+\tau\in \chop{T}\cap (\chop{\+\s})^\preceq$. Then $T(\chop{\+\s})\subseteq T(\+\tau)$ and $\+\s\in T$ imply that $(\+\tau,\s_n)\in T$ and witnesses that $\+\tau\in \chop{R}$. Finally we also observe that for $\+\tau\in \chop{R}$ we have $R(\+\tau) = T(\+\tau)\cap (\s_n)^\preceq$.
\end{proof}

The definition of an $h$-bounded (and so of a computably bounded) tree system is as expected. If~$T$ is computable and computably bounded then for all $k<\length{T}$, $\dom_k T$ is computable and the map $\+\tau\mapsto T(\+\tau)$ is computable.

\smallskip

A \emph{leaf} of a forest system~$T$ is a $\preceq$-maximal element of~$T$. A tuple~$\+\tau$ is a leaf of~$T$ if and only if~$\chop{\+\tau}$ is a leaf of~$\chop{T}$ and~$\tau_{\length{T}}$ is a leaf of $T(\chop{\+\tau})$. The set of leaves of a forest system is prefix-free.

If~$T$ and~$S$ are length~$n$ forest systems then we say that~$T$ is an \emph{end-extension} of~$S$ if:
\begin{itemize}
	 \item $\chop{T}$ is an end-extension of~$\chop{S}$;
	 \item If $\+\tau\in \chop{S}$ is not a leaf of $\chop{S}$ then $T(\+\tau)= S(\+\tau)$;
	 \item If $\+\tau$ is a leaf of $\chop{S}$ then $T(\+\tau)$ is an end-extension of~$S(\+\tau)$.
\end{itemize}
Note that this is a transitive relation.

\begin{lemma} \label{lem:n_step:infinite_union_of_tree_systems}
	Let~$\seq{S_m}$ be a sequence of forest systems above~$A$, with each~$S_{m+1}$ an end-extension of~$S_m$. Then $\bigcup_m S_m$ is a forest system above~$A$. 
\end{lemma}

\begin{proof}
	Let $S = \bigcup_m S_m$. Then $\chop{S} = \bigcup_m \chop{S_m}$, and so by induction on the length, $\chop{S}$ is a forest system above~$\chop{A}$. Let~$\+\tau\in \chop{S}$. Then $S(\+\tau) = \bigcup_m S_m(\+\tau)$ is the union of a sequence of end-extensions above~$A(\+\tau^{-\cchop{A}})$, and so is a forest above that set; note that if~$\+\tau\in \chop{S_m}$ but is not a leaf of~$\chop{S_m}$ then $S(\+\tau) = S_m(\+\tau)$.
\end{proof}

\subsubsection*{Other breaking points} 
\label{ssub:other_breaking_points}

We don't have to isolate only the last coordinate. For example:
\begin{lemma} \label{lem:other_breaking_points:prefix-free}
	Let~$A\subseteq \Strings{n}$. The following are equivalent:
	\begin{enumerate}
		\item $A$ is prefix-free;
		\item For some $k\in \{1,\dots, n-1\}$, $\dom_k A$ is prefix-free and for all $\+\tau\in \dom_k A$, $A(\+\tau)$ is prefix-free; and
		\item For all $k\in \{1,\dots, n-1\}$, $\dom_k A$ is prefix-free and for all $\+\tau\in \dom_k A$, $A(\+\tau)$ is prefix-free.
	\end{enumerate}
\end{lemma}

The proof relies on the fact that $(\chop{A})(\+\tau) = \chop{(A(\+\tau))}$, and induction. For forest systems we do not get as nice a result.

\begin{lemma} \label{lem:other_breaking_points:forest-systems}
	Let~$A\subset \Strings{n}$ be prefix-free and let $T\subseteq A^\preceq$.
	\begin{enumerate}
		\item Suppose that~$T$ is a forest system above~$A$. Then for all $k\in \{1,2,\dots, n-1\}$: (a) $\dom_k T$ is a forest system above~$\dom_k A$; (b) For all $\+\tau\in \dom_k T$, $T(\+\tau)$ is a forest system above	$A(\+\tau^{-\dom_k A})$; and (c) if $\+\tau\preceq \+\tau'$ are in $\dom_k T$ then $T(\+\tau)\subseteq T(\+\tau')$.
		\item Let $k\in \{1,2,\dots, n-1\}$; suppose that $\dom_k T$ is a forest system above~$\dom_k A$, that for all $\+\tau\in \dom_k T$, $T(\+\tau)$ is a forest system above $A(\+\tau^{-\dom_k A})$, and that if $\+\tau\preceq \+\tau'$ are in $\dom_k T$ then $T(\+\tau)$ is an end-extension of $T(\+\tau')$. Then~$T$ is a forest system above~$A$.
	\end{enumerate}
\end{lemma}

Again the proof is routine. In the situation of~(1) we don't always get that $T(\+\tau')$ end-extends $T(\+\tau)$. Suppose for example that $\tau\prec \tau'$ are in $\dom_1 T$ and that $\rho\prec \rho'$ are in $\dom_1 T(\tau)$ (and so also in $\dom_1 T(\tau')$). It is possible that $T(\tau',\rho)\ne T(\tau,\rho)$, even though~$\rho$ is not a leaf of~$T(\tau)$. For example we could have $T(\tau',\rho') = T(\tau',\rho) = T(\tau,\rho')$ which is a proper end-extension of $T(\tau,\rho)$. For end-extending, though, we do get full invariance of breaking point:
\begin{lemma} \label{lem:other_breaking_points:end-extensions}
	Let~$S$ and~$T$ be forest systems of length~$n$. The following are equivalent:
	\begin{enumerate}
		\item $T$ is an end-extension of~$S$;
		\item For some $k\in \{1,\dots, n-1\}$, $\dom_k T$ is an end-extension of~$\dom_k S$, for all $\+\tau\in \dom_k S$, $T(\+\tau)$ is an end-extension of $T(\+\tau)$, and if $\+\tau\in \dom_k S$ is not a leaf of $\dom_k S$, then $T(\+\tau) = S(\+\tau)$.
		\item For all $k\in \{1,\dots, n-1\}$, $\dom_k T$ is an end-extension of~$\dom_k S$, for all $\+\tau\in \dom_k S$, $T(\+\tau)$ is an end-extension of $T(\+\tau)$, and if $\+\tau\in \dom_k S$ is not a leaf of $\dom_k S$, then $T(\+\tau) = S(\+\tau)$.
	\end{enumerate}
\end{lemma}

Also note that if~$S$ is a forest system then $\+\tau\in S$ is a leaf of~$S$ if and only if for some (all) $k\in \{1,2,\dots, \length{S}-1\}$, $\+\tau\rest{k}$ is a leaf of~$\dom_k S$ and $\+\tau\rest{(k,\length{S}]}$ is a leaf of $S(\+\tau\rest{k})$.


\subsubsection*{Paths of tree systems} 
\label{ssub:n_step:paths_of_tree_systems}

We simplify our presentation by restricting ourselves to balanced tree systems.

\begin{definition} \label{def:n_step:balanced}
	Let~$T$ be a tree system and let $m<\w$. We say that~$m$ is a \emph{balanced level of~$T$} if for all $\tau\in \dom_1 T$ of length~$m$, every component of every leaf of~$T(\tau)$ has length~$m$. We say that~$T$ is \emph{balanced} if $\dom_1 T$ has no leaves and~$T$ has infinitely many balanced levels.
\end{definition}

For a balanced tree system~$T$ we let
\[
	[T] = \left\{ \+x\in \Baire{\length{T}} \,:\,  \+x\rest{m}\in T \text{ for every balanced level~$m$ of }T \right\}.
\]
The set~$[T]$ is a closed subset of~$\Baire{n}$.

\smallskip

For $\+x\in [\chop{T}]$ we let
$ T(\+x) = \bigcup_{\+\tau\prec \+x} T(\+\tau).$ This is a tree with no leaves. If~$T$ is balanced then so is $\chop{T}$, and $[T] = \left\{ (\+x,y) \,:\,  \+x\in [\chop{T}] \andd y\in [T(\+x)] \right\}$. If~$T$ is balanced, computable and computably bounded then~$[T]$ is effectively closed.


\subsubsection*{Bushiness for forest systems} 
\label{ssub:step_n:bushiness_for_tree_systems}

Let $\+g = (g_1,\dots, g_n)$ be a tuple of bounding functions, and let~$T$ be a length~$n$ forest system. We say that~$T$ is $\+g$-bushy if $\chop{T}$ is $\chop{\+g}$-bushy and for all $\+\tau\in \chop{T}$, $T(\+\tau)$ is $g_n$-buhsy. As usual, $T$ is $\+g$-bushy if and only if for some (all) $k\in \{1,2,\dots, n-1\}$, $\dom_k T$ is $\+g\rest{k}$-bushy and for all $\+\tau\in \dom_k T$, $T(\+\tau)$ is $\+g\rest{(k,n]}$-bushy.

We say that a set $B\subseteq \Strings{n}$ is \emph{$\+g$-big} above some finite prefix-free set~$A\subset \Strings{n}$ if there is a $\+g$-bushy finite forest system~$R$ above~$A$ whose leaves lie in~$B$. This is extended to all sets~$A$ as above. For $k<n$, $B\subseteq \Strings{n}$, a finite, prefix-free $A\subseteq \Strings{n}$ and an $(n-k)$-tuple~$\+h$ of bounding functions we let
\[
	\project{\+h}{A}{B} = \left\{ \+\tau\in (\dom_k A)^\preceq \,:\, B(\+\tau) \text{ is $\+h$-big above }A(\+\tau^{-\dom_k A})  \right\}.
\]
Note that this notation is different from the one used in the previous section; however, if~$A$ is a singleton $\+\s$ then we revert to the old notation and write $\project{\+h}{\+\s\rest{(k,n]}}{B}$ instead of $\project{\+h}{\+\s}{B}$. A set~$B$ is $\+g$-big above~$A$ if and only $\project{\+g\,\,\rest{(k,n]}}{A}{B}$ is $\+g\rest{k}$-big above~$A$. The proof of this follows the proof of \cref{lem:two_step:bigness}, using \cref{lem:other_breaking_points:forest-systems}(2) (and the fact that every finite prefix-free set is a forest system above itself, and any forest system~$R$ above~$A$ is an end-extension of~$A$).
The proof gives the analogue of \cref{rmk:two_step:minimal_bigness_witness}: if~$B$ is $\+g$-big above~$A$, $T$ is a forest system and $A,B\subseteq T$ then a finite forest system~$S$ witnessing the largeness can be taken to be a subset of~$T$.

\begin{remark} \label{rmk:commutativity_of_projection_operation}
	Let $1\le k < m < n$, let $\+\s\in \Strings{m-k}$, $\+\mu\in \Strings{n-m}$, $\+g$ be an $(m-k)$-tuple of bounding functions, and $\+h$ and $(n-m)$-tuple of bounding functions. Let $B\subseteq \Strings{n}$. Then
	\[
		\project{\+g}{\+\s}{\project{\+h}{\+\mu}{B}} = \project{\+g,\+h}{\+\s,\+\mu}{B}.
	\]
\end{remark}

\smallskip

The big subset property holds for largeness over singletons, with the same proof as that of \cref{lem:TwoSteps:BigSubset}.

\medskip

For the weak concatenation property, we will straightaway work within tree systems. But first we discuss concatenations. Suppose that~$S$ is a finite forest system, that~$A$ is the set of leaves of~$S$, and that~$R$ is a forest system above~$A$. Since~$S$ is finite, $\chop{A}$ is the set of leaves of~$\chop{S}$. We then define~$S\conc R$ by letting:
\begin{itemize}
	\item $\chop{(S\conc R)} = \chop{S}\!\!\conc \chop{R}$;
	\item For $\+\tau\in \chop{S}$, not a leaf of $\chop{S}$, we let $(S\conc R)(\+\tau)= S(\+\tau)$;
	\item For $\+\tau\in \chop{R}$ we let $(S\conc R)(\+\tau) = S(\+\tau^{-A})\conc R(\+\tau) = S(\+\tau^{-A})\cup R(\+\tau)$.
\end{itemize}
Then $S\conc R$ is an end-extension of~$S$, whose leaves are the leaves of~$R$. Also note that if $S,R\subseteq T$ for some forest system~$T$ then $S\conc R\subseteq T$. If both~$S$ and~$R$ are~$\+g$-bushy then so is $S\conc R$. We thus get the restricted analogue of \cref{lem:two_step:concatenating_forest_systems_and_bigness}. From now we fix a forest system~$T$.

\begin{itemize}
	\item Suppose that $B$ is $\+g$-big above~$A$, and that~$C$ is $\+g$-big above~$B$. Then~$C$ is $\+g$-big above~$A$. If $A,B,C\subseteq T$ then every forest system~$S\subseteq T$ witnessing that~$B$ is $\+g$-big above~$A$ has an end-extension $R\subseteq T$ which witnesses that~$C$ is $\+g$-big above~$A$.
\end{itemize}

We get an analogue of \cref{lem:two_step:one_step_induction_for_weak_concatenation}. The notion of an open subset of~$T$ is as expected.

\begin{lemma} \label{lem:n_step:analogue_of_one_step_induction}
	Let~$\mathcal B$ be a finite family of subsets of~$T$ which are open in~$T$. Let~$A\subseteq T$ be finite and prefix-free. Suppose that each~$B\in \mathcal B$ is $\+g$-big above~$A^\preceq\cap T$ (recall that this means that it is $\+g$-big above every finite, prefix-free subset of $A^\preceq \cap T$). Then $\bigcap\mathcal B$ is $\+g$-big above~$A$.
\end{lemma}

We can now prove the analogue of \cref{lem:two_step:last_step_for_weak_concatenation}.
\begin{lemma} \label{lem:n_step:weak_concatenation_the_meat}
	Let~$T$ be a forest system and let~$A,B\subseteq T$; suppose that~$B$ is open in~$T$. Suppose that for all $\+\tau\in A^\preceq\cap T$,~$B$ is $\+g$-big above~$\+\tau$. Then~$B$ is $\+g$-big above~$A$.
\end{lemma}

\begin{proof}
	By induction on the length of~$T$. We may assume that~$A$ is finite and prefix-free. We need to show that $C = \project{g_n}{A}{B}$ is $\chop{g}$-big above~$\chop{A}$.  Let~$\+\tau\in (\chop{A})^\preceq\cap \chop{T}$. We claim that~$C$ is $\chop{\+g}$-big above~$\+\tau$ (and then apply the induction hypothesis). Let~$\+\s = \+\tau^{-\cchop{A}}$. Then $C\cap \+\s^\preceq$ equals $\bigcap_{\mu\in A(\+\s)}\project{g_n}{\mu}{B}$. By assumption, each $\project{g_n}{\mu}{B}$ is $\chop{g}$-big above every tuple in $\+\s^\preceq\cap \chop{T}$; we apply the analogue of \cref{lem:two_step:one_step_induction_for_weak_concatenation} mentioned above.
\end{proof}

\begin{corollary} \label{cor:n_step:weak_concatenation_property}
	Let~$T$ be a tree system, let~$A,B,C\subseteq T$, and suppose that~$C$ is open in~$T$. Suppose that~$B$ is $\+g$-big above~$A$, and that~$C$ is $\+g$-big above every tuple in $B^\preceq \cap T$. Then~$C$ is $\+g$-big above~$A$, and in fact every finite $\+g$-bushy forest system $S\subseteq T$ which witnesses that~$B$ is $\+g$-big above~$A$ has an end-extension $R\subseteq T$ which witnesses that~$C$ is $\+g$-big above~$A$.
\end{corollary}

As a corollary we get the analogue of \cref{lem:two_step:can_take_full_subtree}:
\begin{itemize}
	\item If~$T$ is a bounded and balanced $\+b$-bushy tree system above~$\+\s$, and $B\subset T$ is open in~$T$ and $\+b$-small above~$\+\s$, then for every~$m$ there is some $\+\tau\in T$ such that $|\tau_i|\ge m$ for all $i\le \length{T}$, and above which~$B$ is $\+b$-small.
\end{itemize}

\subsection{The notion of forcing and restriction maps} 
\label{sub:n_step:the_notion_of_forcing_and_the_generic}

We let~$B_{\DNC^n}$ be the set of tuples $\+\tau\in \Strings{n}$ such that either $\chop{\+\tau}\in B_{\DNC^{n-1}}$, or $\tau_n \in B_{\DNC^{\cchop{\+\tau}}}$, that is, if there is some~$e<|\tau_n|$ such that $\tau_n(e) = J^{\cchop{\+\tau}}(e)$.

For brevity, for a tuple~$\+\s\in \Strings{n}$ we let $|\+\s| = \min \left\{ |\s_i| \,:\,  i\le n \right\}$. When a tuple-length~$n$ is clear from the context, then for a function~$g$ we let $\+g = (g,g,\dots, g)$.

\smallskip

We let~$\PP_n$ be the set of tuples $\pp = (\+{\s}^\pp,T^\pp,B^\pp,h^\pp,b^\pp)$ satisfying:
\begin{enumerate}
	\item $T^\pp$ is a computably bounded, computable, balanced tree system above~$\+{\s}^\pp$;
	\item $h^\pp\in \Quick$ and~$T^\pp$ is~$\+{h}^\pp$-bushy;
	\item $B^\pp\subset T^\pp$ is c.e.\ and open in~$T^\pp$, and $B^\pp\supseteq B_{\DNC^n}\cap T^\pp$;
	\item $b^\pp\in \Quick$ and $B^\pp$ is $\+{b}^\pp$-small above~$\+{\s}^\pp$; and
	\item $h^\pp\gg b^\pp$ and $h^\pp\ge b^\pp$ above $|\+{\s}^\pp|$.
\end{enumerate}

We define a partial ordering on~$\PP_n$ as follows. A condition~$\qq$ extends a condition~$\pp$ if $\+{\s}^\pp\preceq \+{\s}^\qq$, $T^\qq$ is a subsystem of~$T^\pp$, $B^\pp \cap T^\qq\subseteq B^\qq$, and $h^\qq\le h^\pp$ and $b^\qq\ge b^\pp$ above $|\+{\s}^\qq|$.

\smallskip

The assignment of closed sets $X^\pp = [T^\pp]\setminus [B^\pp]^\prec$ for $\pp\in \PP_n$ is acceptable; the proof is identical to the proof of \cref{lem:two_step:acceptability_of_closed_sets}.

If $\GG\subset \PP_n$ is sufficiently generic then we denote the generic tuple (the element of the singleton $\bigcap_{\pp\in \GG} [T^\pp]\setminus [B^\pp]^\prec$) by $\+x^\GG$. As above, every condition in~$\PP_n$ forces that~$x_n^\GG$ is $\DNC$ relative to $\chop{\+x^\GG}$.

\subsubsection*{The restriction maps} 
\label{ssub:n_step:the_restriction_maps}

For all $n\ge 2$, define $i_n\colon \PP_n\to \PP_{n-1}$ by letting
\[
	i_n(\qq) = \big( \chop{\+\s^\qq}, \chop{T^\qq},
	\project{b^\qq}{\s^\qq_n}{B^\qq},  h^\qq,b^\qq \big),
\]
where we have
\[
	\project{b^\qq}{\s^\qq_n}{B^\qq} =
	 \left\{ \+\tau\in \chop{T^\qq}  \,:\, B^\qq(\+\tau) \text{ is $b^\qq$-big above }\s_n^\qq   \right\}.
\]
It is routine to check that $i_n(\qq)\in \PP_{n-1}$ for all $\qq\in \PP_n$. Inductively we define $\QQ_n\subset \PP_n$: $\QQ_1 = \PP_1$, and $\QQ_n$ is the set of conditions $\qq\in \QQ_n$ such that:
\begin{itemize}
	\item $i_n(\qq)\in \QQ_{n-1}$; and
	\item $\project{b^\qq}{\+\s^\qq}{B^\qq} = \left\{ \+\tau\in \chop{T^\qq} \,:\,  \s^\qq_n\in B^\qq(\+\tau) \right\}$.
\end{itemize}
We again observe that for all $\qq\in \QQ_n$, $\chop{X^\qq} = X^{i_n(\qq)}$; the proof is the same as above. The proof that the restriction of~$i_n$ to~$\QQ_n$ is order-preserving is identical to that in the proof of \cref{lem:TwoStep:the_restriction_map}.

\begin{lemma} \label{lem:n_step:the_dense_suborderings}
	There is a map $\nu_n\colon \PP_n\to \QQ_n$ such that:
	\begin{enumerate}
		\item $\nu_n(\qq)\le \qq$ for all $\qq\in \PP_n$; and
		\item $i_n\circ \nu_n = \nu_{n-1}\circ i_n$.
	\end{enumerate}
\end{lemma}

In particular, $\QQ_n$ is dense in~$\PP_n$.

\begin{proof}
	We omit the indices~$n$ and	$n-1$ from $i_n$, $\nu_n$ etc.; they will be clear from the context.

	Let $\qq\in \PP_n$. For brevity we let $C_n = B^\qq$ and for $k\in \{1,\dots, n-1\}$ we let $C_k = \project{\+b^\qq}{\+\s^\qq\,\,\rest{(k,n]}}{B^\qq}$. \Cref{rmk:commutativity_of_projection_operation} says that if $k<m\le n$ then $C_k = \project{\+b^\qq}{\+\s^\qq\,\,\rest{(k,m]}}{C_m}$.

	We define a tuple $\nu(\qq) = (\+\s^\qq,T^\qq,B^{\nu(\qq)}, h^\qq,b^\qq)$ by letting
	\[
		B^{\nu(\qq)} = \left\{ \+\tau\in T^\qq \,:\,  \+\tau\rest{k}\,\in C_k \text{ for some }k\le n \right\}.
	\]

	The set	$B^{\nu(\qq)}$ is $\+b^\qq$-small above~$\+\s^\qq$. For let~$D$ be the set of leaves of a $\+b^\qq$-bushy finite tree system~$S\subset T^\qq$ above~$\+\s^\qq$. Since $C_1$ is $b^\qq$-small above~$\s^\qq_1$ we find some $\tau_1\in (\dom_1 D)\setminus C_1$. Since $C_1 = \project{b^\qq}{\s_2^\qq}{C_2}$, $C_2(\tau)$ is $b^\qq$-small above~$\s^\qq_2$; we find some~$\tau_2$ such that $(\tau_1,\tau_2)\in (\dom_2 D) \setminus C_2$; and so on, we find some $\+\tau\in D\setminus B^{\nu(\qq)}$. We conclude that $\nu(\qq)\in \PP_n$ (and $\nu(\qq)\le \qq$).

	Now $B^{i(\qq)} = C_{n-1}$; so $B^{\nu(i(\qq))}$ is the set of tuples $\+\tau\in \chop{T^\qq}$ such that $\+\tau\rest{k}\,\in C_k$ for some $k\le n-1$.

	Let~$\+\tau\in T^\qq$. If $\chop{\+\tau}\in B^{\nu(i(\qq))}$ then $B^{\nu(\qq)}(\+\tau)= T^\qq(\+\tau)$, in particular $\s^\qq_n\in B^{\nu(\qq)}(\+\tau)$. Otherwise, $B^{\nu(\qq)}(\+\tau)= B^\qq(\tau)$, and since in this case $\+\tau\notin C_{n-1}$ we see that $B^{\nu(\qq)}(\+\tau)$ is $b^\qq$-small above $\s^\qq_n$. We conclude that $B^{i(\nu(\qq))} = \project{b^\qq}{\s^\qq_n}{B^{\nu(\qq)}} = B^{\nu(i(\qq))}$ and so that $i(\nu(\qq)) = \nu(i(\qq))$.

	We also conclude that $\+\tau\in \project{b^\qq}{\s^\qq_n}{B^{\nu(\qq)}}$ if and only if $\s^\qq_n\in B^{\nu(\qq)}(\+\tau)$. By induction, $\nu(i(\qq))\in \QQ_{n-1}$, so $\nu(\qq)\in \QQ_n$.
\end{proof}

\begin{proposition} \label{prop:n_step:restriction_maps}
	$i_n\rest{\QQ_n}$ is a restriction map from~$\QQ_n$ to~$\QQ_{n-1}$.
\end{proposition}

\begin{proof}
	It remains to show that if $\qq\in \QQ_n$ and $\pp\in \QQ_{n-1}$ extends $i_n(\qq)$ then there is some $\rr\in \QQ_n$ extending~$\qq$ such that $i_n(\rr)\le \pp$. By using the map~$\nu_n$, it suffices to find~$\rr\in \PP_n$. The proof is identical to that of \cref{lem:TwoStep:the_restriction_map}.
\end{proof}

\begin{lemma} \label{lem:n_step:restriction_is_onto}
	$i_n\rest{\QQ_n}$ is onto~$\QQ_{n-1}$.
\end{lemma}

\begin{proof}
	Let $\pp\in \QQ_{n-1}$. We define $\qq\in \QQ_n$ such that $i_n(\qq)=\pp$ by letting, for $\+\s\in T^\pp$, $T^\qq(\+\s) = (h^\pp)^{\le |\+\s|}$, and let $B^\qq(\+\s) = T^\qq(\+\s)$ if $\+\s\in B^\pp$, otherwise $B^\qq(\+\s) = B_{\DNC^{\+\s}}$. 
\end{proof}


\subsubsection*{Totality} 
\label{ssub:n_step:totality}

\begin{proposition} \label{prop:n_step:totality}
	Let~$C\subseteq \Baire{n}$ be~$\Pi^0_2$ and let~$\pp\in \PP_n$. If $\pp\force \+x^\GG\in C$ then~$\pp$ has an extension which strongly forces that~$\+x^\GG\in C$.
\end{proposition}

The proof is identical to the proof of \cref{prop:two_step:forcing_Pi_2}.



\subsection{Minimality} 
\label{sub:step_n:minimality}

Let~$\Gamma\colon \Baire{n}\to 2^\w$ be a Turing functional.

\begin{definition} \label{def:step_n:local-splitting}
	Let $B\subseteq \Strings{n}$. Two sets $A_0,A_1\subset \Strings{n}$ form a \emph{local $\Gamma$-splitting mod~$B$} if for all $\+\tau\in \Strings{n-1}$, the sets $A_0(\+\tau)$ and $A_1(\+\tau)$ $\Gamma(\+\tau,-)$-split mod $B(\+\tau)$.
\end{definition}

\begin{definition} \label{def:n_step:uniform_bigness}
	Let~$A\subset \Strings{n}$ be finite and prefix-free, and let $\mathcal B$ be a collection of subsets of $\Strings{n}$. We say that the sets in~$\mathcal B$ are \emph{uniformly $\+g$-big above~$A$} if $\bigcap_{B\in \mathcal B} \project{g_n}{A}{B}$ is $\chop{\+g}$-big above~$\chop{A}$.
\end{definition}

\begin{lemma} \label{lem:step_n:finding_local_splittings}
	Suppose that $\pp\in \PP_n$ strongly forces that $\Gamma(\+x^\GG)$ is total, and forces that it is not computable from $\chop{\+x^\GG}$. Let $\+\s\in T^\pp$; let $g\in \Quick$ such that $h^\pp\gg g$, and $h^\pp\ge 3g$ and $ g\ge b^\pp$ above $|\+\s|$. Then there are sets $A_0,A_1\subset T^\pp$, uniformly $\+g$-big above~$\+\s$, which locally $\Gamma$-split mod~$B^\pp$.
\end{lemma}

\begin{proof}
	Identical to the proof of \cref{lem:two_steps:finding_local_splittings}.
\end{proof}

\begin{lemma} \label{lemma:n_steps:combinatorial}
	Let~$\+g$ and $\+h$ be $n$-tuples of bounding functions; let~$B\subseteq \Strings{n}$ be open. Suppose  that:
	\begin{itemize}
		\item $\+\s,\+\s^*\in \Strings{n}$;
		\item $A$ is $3\+g$-big above~$\+\s$;
		\item $E_0$ and~$E_1$ are uniformly $3\+g$-big above~$A$; and for all $\+\tau\in A$, $E_0\cap \+\tau^\preceq$ and $E_1\cap \+\tau^\preceq$ locally~$\Gamma$-split mod~$B$; and
		\item $F$ is $3\+h$-big above $\+\s^*$, and $|\Gamma(\+\rho)|> |\Gamma(\+\zeta)|$ for all $\+\rho\in F\setminus B$ and all $\+\zeta\in E\setminus B$, where $E = E_0\cup E_1$.
	\end{itemize}
	Then there are $E'\subseteq E$, $\+g$-big above~$\+\s$, and $F'\subseteq F$, $\+h$-big above~$\+\s^*$, which $\Gamma$-split mod~$B$.
\end{lemma}		

\begin{proof}
	Identical to the proof of \cref{lemma:two_steps:combinatorial}.
\end{proof}

\begin{lemma} \label{lem:main_minimality_lemma}
	Suppose that $\pp\in \PP_n$ strongly forces that $\Gamma(\+x^\GG)$ is total, and forces that it is not computable from $\chop{\+x^\GG}$. Let $k\in \{0,1,\dots, n-1\}$. Let $C\subset T^\pp$ be finite and prefix-free. Let~$g\in \Quick$ such that $h^\pp\gg g$, and $h^\pp\ge 3^{|C|^2}g$ and $ g\ge b^\pp$ above $|\+\s|$ for all $\+\s\in C$.

	Then there is a set~$A\subset T^\pp$, $\+g$-big above~$C$, such that for all $\+\tau\in \dom_k A$, the sets in the collection
	\[
		\left\{ A(\+\tau)\cap (\rho,\+\s^\pp\rest{(k+1,n]})^\preceq \,:\,  \rho\in \dom_1 A(\+\tau) \right\}
	\]
	pairwise $\Gamma(\+\tau,-)$-split mod~$B^\pp(\+\tau)$.
\end{lemma}

\begin{proof}
	The notation for the case $k=0$ is slightly easier. In this case we closely follow the proof of \cref{lem:two_steps:getting_many_global_splittings}. For simplicity of notation, for a set $A\subseteq T^\pp$ and some tuple $\+\tau\in \dom_k T^\pp$ (for some $k<n$) we let $A\cap (\+\tau)^\preceq = A\cap (\+\tau,\+\s^\pp\rest{(k,n]})^\preceq$. We prove the lemma by induction on~$|C|$; we let $C^* = C\cup \{\+\s^*\}$; by induction we are given $A$ which is $3^{|C|}g$-big above~$C$, and the sets $A\cap (\rho)^\preceq$ (for $\rho\in \dom_1 C$) pairwise~$\Gamma$-split mod~$B^\pp$. We list the elements $\+\s_1,\+\s_2,\dots, \+\s_m$ of~$C$ such that $(\s_j)_1\ne \s^*_1$. By reverse recursion on $j\le m$ we define sets~$A_j\subset T^\pp$ with $A_{j-1}\subset A_j^\preceq$ and $A_m\cap \+\s^\preceq \subset A^\preceq$ for all $\+\s\in C$. We ensure that~$A_j$ is $3^j\+g$-big above~$C^*$ and that  $A_{j-1}\cap \+\s_j^\preceq$ and $A_{j-1}\cap \+\s^*$ form a $\Gamma$-splitting mod~$B$.

	We start with $A_m = A \cup \{\+\s^*\}$. Say we are given~$A_j$, $j>0$. For brevity let~$D_j = \chop{(A_j \cap \+\s_j^\preceq)}$.
	For $\+\tau\in A_j \cap \+\s_j^\preceq$ we let $Q_{\+\tau}$ be the set of $\+\zeta\in {D_j}^\preceq\cap \chop{T^\pp}$ such that either:
	\begin{itemize}
		\item $\chop{\+\tau}\npreccurlyeq \+\zeta$; or
		\item in $T^\pp(\+\zeta)$ there are $G_0$ and~$G_1$, $3^jg$-big above~$\tau_n$, which $\Gamma(\+\zeta,-)$-split mod $B^\pp(\+\zeta)$.
	\end{itemize}
	Then \cref{lem:step_n:finding_local_splittings} says that for all $\+\mu\in D_j^\preceq \cap \chop{T^\pp}$ the set $Q_{\+\tau}$ is $3^j\+g$-big above~$\+\mu$. By \cref{lem:n_step:weak_concatenation_the_meat}, $Q_{\+\tau}$ is $3^j\+g$-big above $D_j^\preceq \cap \chop{T^\pp}$. By \cref{lem:n_step:analogue_of_one_step_induction}, $\bigcap_{\+\tau\in A_j\cap \+\s_j^\preceq} Q_{\+\tau}$ is $3^j\+g$-big above~$D_j$. Thus, we can find $E_{j,0}$ and~$E_{j,1}$, finite subsets of~$T^\pp$ which are uniformly $3^j\+g$-big above~$A_j\cap \+\s_j^\preceq$, which locally $\Gamma$-split mod~$B^\pp$. We obtain~$F_j$ as before. Applying \cref{lemma:n_steps:combinatorial} we finally get $F'_j\subset A_j^\preceq \cap (\+\s^*)^\preceq$, $3^{j-1}\+g$-big above~$\+\s^*$, and $E'_j\subset A_j^\preceq \cap \+\s_j^\preceq$, $3^{j-1}\+g$-big above~$\+\s_j$, which $\Gamma$-split mod~$B^\pp$.

	In this proof we emply the following notation: for a set $X\subset \Strings{n}$ and $k\le n$ we let $X_k = \dom_k X$. To define a set~$X$ it suffices to first define $X_1$; then, for all $\tau_1\in X_1$, define $X_2(\tau_1)$ (a set of strings); then, for all $(\tau_1,\tau_2)\in X_2$, define $X_3(\tau_1,\tau_2)$, and so on.

	We define the set~$A_{j-1}$. First, we consider all $\+\s\in C^*$ such that $\s_1\ne \s^*_1,(\s_j)_1$. For all such $\+\s$ we let $A_{j-1}\cap \+\s^\preceq = A_j\cap \+\s^\preceq$. We let $A_{j-1,1}\cap (\s^*_1)^\preceq = (F'_j)_1$ and $A_{j-1,1}\cap ((\s_j)_1)^\preceq = (E'_j)_1$.

	Next, consider all $\+\s\in C^*$ such that $\s_1 = \s^*_1$, but $\s_2\ne \s^*_2$. For all $\tau_1\in (F'_j)_1$ we let $A_{j-1}(\tau_1)\cap (\+\s\rest{(1,n]}) = A_j(\tau_1^{-A_{j,1}})\cap (\+\s\rest{(1,n]})$; this completely defines $A_{j-1}\cap \+\s^\preceq$. We similarly define $A_{j-1}\cap \+\s^\preceq$ for $\+\s\in C^*$ such that $\s_1 = (\s_j)_1$ but $\s_2\ne (\s_j)_2$. Then, for all $\tau_1\in (F'_j)_1$ we let $A_{j-1,2}(\tau_1) = (F'_j)(\tau_1)$; this defines $A_{j-1,2}\cap (\+\s^*\rest{2})^\preceq$, and similarly define $A_{j-1,2}\cap (\+\s_j\rest{2})^\preceq$. The process continues similarly until all of~$A_{j-1}$ is defined.

	\bigskip

	The case $k>0$ is very similar. Morally it follows the idea of the proof of \cref{lem:two_steps:finding_many_local_splittings}, extending bushily on the first~$k$ coordinates so that we can emulate the proof of the case $k=0$ (but with $n-k$ replacing~$n$) within the image. We give a sketch. Again we work by induction on~$|C|$; we start with some~$C$ for which we inductively already have~$A$ as required; and add to~$C$ a tuple~$\+\s^*$ to get~$C^*$. We now let the list~$\+\s_1,\+\s_2,\dots, \+\s_m$ contain those elements $\+\s\in C$ such that $\+\s\rest k = \+\s^*\rest{k}$ but $\s_{k+1}\ne \s^*_{k+1}$. We start with $A_m = A\cup \{\+\s^*\}$ and build sets~$A_j$ with the same properties as above. Given~$A_j$ we aim to find~$E_{j,0},E_{j,1}$ and~$F_j$ as above, except that we also require that $\dom_k E_j = \dom_k F_j$; this is possible because $\+\s_j\rest{k} = \+\s^*\rest{k}$: we first get~$E_j$ as above, and then extend $\dom_k E_j$ to $\dom_k F_j$; and ``relabel'' $E_j$ by letting $E_j(\zeta) = E_j(\tau)$ for all $\zeta\in \dom F_j$ extending $\tau\in \dom E_j$. Then we obtain $E'_j$ and~$F'_j$ but require that $\dom_k E'_j = \dom_k F'_j = \dom F_j$; we apply \cref{lemma:n_steps:combinatorial} within $T^\pp(\+\zeta)$ for each $\+\zeta\in \dom F_j$. We then define $A_{j-1}$ as above.
\end{proof}

\begin{proposition} \label{prop:n_step:final_straw}
	Every condition in~$\PP_n$ forces that $\deg_\Tur(\+x^\GG)$ is a strong minimal cover of $\deg_\Tur(\chop{\+x^\GG})$.
\end{proposition}

\begin{proof}
	Let $\pp\in \PP_n$ which strongly forces that~$\Gamma(\+x^\GG)$ is total, and forces that it is not computable from $\chop{\+x^\GG}$. Fix $k\in \{0,1,\dots, n-1\}$. Using \cref{lem:main_minimality_lemma} and the by now familiar construction we obtain an extension~$\qq$ of~$\pp$ which (strongly) forces that $\Gamma(\+x^\GG)\oplus (\+x^\GG\rest{k})\ge_\Tur x^\GG_{k+1}$. Iterating for each~$k$ we obtain a condition which forces that $\Gamma(\+x^\GG) \equiv_\Tur \+x^\GG$.
\end{proof}


\section{Proof of the main theorem} 
\label{sec:proof_of_the_main_theorem}

We prove \cref{thm:main}. We have obtained directed sequence of forcing notions

\begin{center}
	\begin{tikzpicture}
	\foreach \i/\j in {1/0,2/2,3/4,4/6}	
		\node (Q\i) at (\j,0) {$\QQ_\i$};
	\foreach \i/\j in {1/2,2/3,3/4}
	\draw[<-,>=latex] (Q\i) -- node[above] {$i_\j$} (Q\j);
	\draw[<-,>=latex] (Q4) -- node[above] {$i_5$} (8,0);
	\node at (8.5,0) {$\cdots$};
\end{tikzpicture}
\end{center}

With each~$i_n$ a restrction map. For $m<n$ let $i_{n\to m} = i_{m+1}\circ i_{m+2}\circ \cdots \circ i_n$ (and of course let $i_{n\to n} = \id_{\QQ_n}$). A composition of restriction maps is a restriction map, so each $i_{n\to m}$ is a restriction map. 

As sets, the forcing notions~$\QQ_n$ are pairwise disjoint. Let $\QQ_{<\w} = \bigcup_n \QQ_n$. We order~$\QQ_{<\w}$ as follows: if $\pp\in \QQ_n$ and $\qq\in \QQ_{m}$ then~$\qq$ extends~$\pp$ if $m\ge n$ and $i_{m\to n}(\qq)\le \pp$ in $\QQ_n$. Note that the ordering on each~$\QQ_n$ agrees with this ordering. 

For~$n< \w$ let $\QQ_{\le n} = \bigcup_{m\le n} \QQ_m$, ordered as a sub-order of~$\QQ_{<\w}$. Define $j_{\w\to n}\colon \QQ_{<\w}\to \QQ_{\le n}$ by letting, for $\qq\in \QQ_m$, $j_{\w\to n}(\qq) = \qq$ if $m\le n$, and otherwise $j_{\w\to n}(\qq) = i_{m\to n}(\qq)$. For $m\ge n$ let $j_{m\to n}\colon \QQ_{\le m}\to \QQ_{\le n}$ be $j_{\w\to n}\rest{\QQ_{\le m}}$. These maps are restriction maps and they commute: for $n\le m \le \alpha\le \w$, $j_{\alpha\to n} = j_{m\to n}\circ j_{\alpha\to m}$. 

Let~$G_{<\w}\subset \QQ_{<\w}$ be very generic. Let $G_{\le n}$ be the filter in~$\QQ_{\le n}$ generated by the generic directed set $j_{\w\to n}[G_{<\w}]$. By \cref{lem:n_step:restriction_is_onto}, each~$\QQ_n$ is dense in~$\QQ_{\le n}$; so $G_n = G_{\le n}\cap \QQ_n$ is a fairly generic filter of~$\QQ_n$; and $i_{m\to n}[G_m] \subseteq G_n$. 

This gives us a sequence $x_1,x_2,\dots$ of elements of Baire space such that $(x_1,\dots, x_n) = \+x^{G_n}$. By \cref{prop:n_step:final_straw}, each tuple $(x_1,\dots, x_n)$ is a strong minimal cover of $(x_1,\dots, x_{n-1})$; and $x_n\in \DNC^{(x_1,\dots, x_{n-1})}$.



\

\bibliographystyle{plain}
\bibliography{DNRchain}

%
%
%
%
%
%

\end{document}